\documentclass[11pt,twoside]{amsart}
\usepackage{lmodern}
\usepackage{amsmath,amssymb,amsfonts,amsthm,mathtools}
\usepackage{mathrsfs}
\usepackage[a4paper,textwidth=15cm,textheight=20.5cm,includehead]{geometry}
\usepackage{xcolor}
\usepackage{cite}
\usepackage[hidelinks]{hyperref}
\hypersetup{
  pdftitle={Holomorphic and Formal First Integrals for Foliations of Codimension One on Complex Analytic Space Germs},
  pdfauthor={V\'ictor Le\'on and Bruno Sc\'ardua}
}
\usepackage{tikz}
\usepackage{tikz-cd}
\usepackage{enumitem}
\usepackage[hyperpageref]{backref}

\usepackage{graphicx}

\usepackage[utf8]{inputenc}
\usepackage[all]{xy}

\newcommand{\reg}{\mathrm{reg}\,}
\newcommand{\C}{\mathbb C}
\newcommand{\N}{\mathbb N}

\newcommand{\Q}{\mathbb Q}

\newcommand{\Pone}{\mathbb P^1}
\newcommand{\OO}{\mathcal O}
\newcommand{\OX}{\mathcal O}
\newcommand{\F}{\mathcal F}
\newcommand{\FF}{\mathcal F}

\newcommand{\J}{\mathcal J}
\newcommand{\T}{\mathcal T}

\newcommand{\mm}{\mathfrak m}

\newcommand{\wh}{\widehat}
\newcommand{\wt}{\widetilde}
\newcommand{\CC}{\mathbb C}
\newcommand{\NN}{\mathbb N}

\DeclareMathOperator{\Sing}{Sing}
\DeclareMathOperator{\codim}{codim}

\DeclareMathOperator{\Frac}{Frac}

\DeclareMathOperator{\Hom}{Hom}
\DeclareMathOperator{\Der}{Der}
\DeclareMathOperator{\sing}{sing}

\DeclareMathOperator{\Diff}{Diff}
\DeclareMathOperator{\Hol}{Hol}

\DeclareMathOperator{\ord}{ord}

\newcommand{\transv}{\mathrel{\text{\tpitchfork}}}
\makeatletter
\newcommand{\tpitchfork}{%
  \vbox{
    \baselineskip\z@skip
    \lineskip-.52ex
    \lineskiplimit\maxdimen
    \m@th
    \ialign{##\crcr\hidewidth\smash{$-$}\hidewidth\crcr$\pitchfork$\crcr}
  }%
}

\newtheorem{lema}{Lemma}
\newtheorem{cor}{Corollary}
\newtheorem{teo}{Theorem}
\newtheorem*{teo1}{Theorem A}
\newtheorem*{teo2}{Theorem B}
\newtheorem*{teo3}{Theorem C}

\newtheorem{claim}{Claim}

\newtheorem{propo}{Proposition}

\newtheorem{mainthm}{Theorem}

\theoremstyle{definition}
\newtheorem{remark}{Remark}

\newtheorem{defi}{Definition}
\newtheorem{exe}{Example}

\usepackage{fancyhdr}

\begin{document}

\pagestyle{fancy}


\title{Holomorphic and Formal First Integrals
for Foliations of Codimension One on
Complex Analytic Space Germs}

\author{V\'ictor Le\'on}
\address{V\'ictor Le\'on\\
ILACVN - CICN, Universidade Federal da Integra\c c\~ao Latino-Americana,\\
ITAIPU PARQUETEC, Foz de Igua\c cu-PR, 85867-900 - Brazil
}
\email[V. Le\'on]{victor.leon@unila.edu.br, vicaml19@gmail.com}

\author{Bruno Sc\'ardua}
\address{Bruno Sc\'ardua\\ Instituto de Matem\'atica - Universidade Federal do Rio de Janeiro,\\
CP 68530, Rio de Janeiro-RJ, 21945-970 - Brazil}
\email[B. Sc\'ardua]
{bruno.scardua@gmail.com}



\begin{abstract}
\sloppy

We study holomorphic and formal first integrals for germs of codimension-one holomorphic foliations on normal complex analytic spaces. In dimension two, under the assumption that the dual graph of the exceptional divisor of a resolution is a tree, we prove that the foliation admits a holomorphic first integral if and only if its leaves are closed outside the singular point and only finitely many leaves accumulate at that point. This extends a classical integrability theorem of Mattei and Moussu to singular ambient spaces. We also prove a holomorphic prolongation theorem for normal quotient germs admitting a smooth quasi-\'etale cover and a smooth connected lift of a generic two-dimensional section. We record, in addition, a conditional formal prolongation statement under depth assumptions on the conormal powers and an injectivity condition for the corresponding differential-form obstruction modules. Under the quotient-prolongation hypothesis, and with a reduced tangent cone where formal restriction must be detected, the higher-dimensional integrability results follow from their surface counterparts. We give a reduced nonnormal example satisfying both dynamical conditions but admitting no holomorphic first integral, showing that normality is essential. Our arguments combine resolution of singularities, holonomy techniques, formal completion, and extension properties of holomorphic functions on normal analytic spaces.

\end{abstract}

\maketitle

\hyphenation{sin-gu-lari-ties}
\hyphenation{re-si-dues}

\tableofcontents


 \section{Introduction}\label{secintro}

This work is dedicated to the study of integrability of germs of holomorphic foliations on complex analytic spaces with singularities.

The case of a nonsingular ambient space was first addressed in the
fundamental work of Mattei--Moussu \cite{MatteiMoussu1980}, where the authors
consider germs of codimension-one holomorphic foliations at the origin
$0\in\C^n$, $n\geq2$.

In this work, we prove versions of some of the main results in \cite{MatteiMoussu1980} for germs of codimension-one holomorphic foliations on germs of normal complex analytic varieties.

Our first main result is:

\begin{mainthm}\label{thm:A}\normalfont
Let $(X,0)$ be a normal-tree complex analytic surface germ, and let $\F$
be a germ of a holomorphic foliation by curves on $(X,0)$. Then the
following conditions are equivalent:
\begin{enumerate}
\item[(i)] $\F$ admits a holomorphic first integral $f \in \OO_{X,0}$.
\item[(ii)] $\F$ has a representative $\F_U$ in a neighborhood $0 \in U \subset X$ such that:
\begin{enumerate}
\item[(ii.a)] the leaves of $\F_U$ are closed subsets of $U\setminus\{0\}$;
\item[(ii.b)] only a finite number of leaves of $\F_U$ adhere to the origin $0\in X$.
\end{enumerate}
\end{enumerate}
\end{mainthm}

Example~\ref{exem2} shows that normality is essential: on a reduced
nonnormal space, both dynamical conditions may hold although no
nonconstant holomorphic first integral exists.

Theorem~\ref{thm:A} is a version, for two-dimensional singular spaces, of Theorem~B on page~473 of \cite{MatteiMoussu1980}.

By a \textit{normal-tree} complex analytic surface, we shall mean a normal
surface germ $(X,0)$ whose exceptional divisor in a good resolution has a
tree as its dual graph. This property is independent of the chosen good
resolution, since further point blow-ups preserve the tree property.

We shall explain this notion in detail later in this paper.

We also prove a codimension-one version of Theorem~A using the
Prolongation Theorem, which can be stated as follows:

\begin{mainthm}\label{thm:B}\normalfont
Let $(X,0)$ be a germ of a normal-tree complex analytic variety of dimension $n \geq 2$, and let $\F$ be a germ of a codimension-one holomorphic foliation on $(X,0)$. If $n\geq3$, assume in addition that $\F$ is defined by a holomorphic integrable $1$-form $\omega$ with $\codim_X\Sing(\F)\geq2$ and that the pair $((X,0),\F)$ is quotient-prolongation-admissible in the sense of Definition~\ref{def:qpa} below. Then the following conditions are equivalent:
\begin{enumerate}
\item[(i)] $\F$ admits a holomorphic first integral $f \in \OO_{X,0}$;
\item[(ii)] there is a representative $\F_U$ of $\F$ in a neighborhood $0\in U\subset X$ such that:
\begin{enumerate}
\item[(ii.a)] the leaves of $\F_U$ are closed subsets of $U-\Sing(\F)$;
\item[(ii.b)] only finitely many leaves of $\F_U$ adhere to the origin $0\in U$.
\end{enumerate}
\end{enumerate}
\end{mainthm}
We also consider the problem of formal integrability in the two-dimensional
setting. Our next result shows that, for germs of codimension-one holomorphic
foliations on normal-tree complex analytic surface germs, the existence of a
formal first integral already implies the existence of a holomorphic first
integral.
\begin{mainthm}\label{thm:C}\normalfont
Let $\F$ be a germ of a holomorphic foliation by curves on a
normal-tree complex analytic surface $(X,0)$. If $\F$ admits a formal
first integral $\hat f\in\widehat\OX_{X,0}$, then $\F$ admits a holomorphic
first integral $f\in\OX_{X,0}$.
\end{mainthm}

We finally address the question of formal integration and prove the following version of \cite{MatteiMoussu1980,Malgrange1976} for the singular case:

\begin{mainthm}\label{thm:D}\normalfont
Let $(X,0)$ be a normal-tree complex analytic germ of dimension $n\geq2$,
and let $\F$ be a germ of a codimension-one holomorphic foliation on
$(X,0)$. If $n\geq3$, assume in addition that $\F$ is defined by a
holomorphic integrable $1$-form $\omega$ with
$\codim_X\Sing(\F)\geq2$, that the pair $((X,0),\F)$ is
quotient-prolongation-admissible, and that the tangent cone $C_0(X)$ is
reduced. If $\F$ admits a nonconstant formal function
$\hat f\in\widehat{\OO}_{X,0}$ satisfying
$\omega\wedge d\hat f=0$ in the completed module of K\"ahler
$2$-forms, then $\F$ admits a holomorphic first
integral $f\in\OO_{X,0}$.
\end{mainthm}

Under the hypotheses stated above, Theorem~\ref{thm:D} is a
singular-analytic-space counterpart of Theorem~A on page~472 of
\cite{MatteiMoussu1980}.
Let us explain its hypotheses.
A germ $(X,0)\subset(\C^m,0)$ of a complex analytic variety of dimension
$n\geq2$ is a \textit{normal-tree} if $(X,0)$ is normal and there is a
nonempty Zariski-open
subset of the Grassmannian of codimension-$(n-2)$ linear subspaces through
$0$ such that every corresponding Bertini-type surface section $(X^*,0)$
is a normal-tree analytic surface.

\begin{defi}\label{def:qpa}
For $n\geq3$, a pair $((X,0),\F)$ is called
\textit{quotient-prolongation-admissible} if there are a smooth germ
$(\widetilde X,\widetilde0)\simeq(\C^n,0)$, a finite group $G$ acting
holomorphically on $\widetilde X$, and a finite quasi-\'etale quotient map
\[
 \pi:(\widetilde X,\widetilde0)\longrightarrow(X,0),
 \qquad X\simeq\widetilde X/G,
\]
together with a nonempty Zariski-open family of generic two-dimensional
sections $Y\subset X$ such that
$\widetilde Y=(\pi^{-1}(Y))_{\mathrm{red}}$ is a smooth connected surface
germ and the saturated pullback foliation $\pi^*\F$ restricts to a foliation
on $\widetilde Y$ with an isolated singularity at $\widetilde0$.
\end{defi}

In the course of the proof of Theorem~\ref{thm:B}, we establish and apply the following theorem:

\begin{mainthm}[Prolongation theorem]\label{thm:E}\normalfont
Let $(X,0)$ be a normal irreducible complex analytic germ of dimension
$n\geq3$, let $\omega$ be a holomorphic integrable 1-form on $(X,0)$ defining a codimension-one foliation $\F$ with
$\codim_X\Sing(\F)\geq2$, and let
$i:(Y,0)\hookrightarrow(X,0)$ be a generic normal surface section such that
$Y\cap(X_{\mathrm{sing}}\cup\Sing(\F))=\{0\}$ and $i^*\F$ has an isolated
singularity. Let
$f_0\in\OO_{Y,0}$ be a nonconstant holomorphic first integral of $i^*\F$.
Suppose that $X\simeq\widetilde X/G$ and that the section $Y$ satisfies
the smooth quasi-\'etale-cover conditions in
Definition~\ref{def:qpa}. Then $f_0$ extends to a holomorphic first integral
$f\in\OO_{X,0}$ satisfying $i^*f=f_0$.
\end{mainthm}

The Prolongation Theorem is a singular quotient-space counterpart of the
extension theorem on page~471 of \cite{MatteiMoussu1980}.

The following theorem complements Theorem~\ref{thm:E} by giving a
formal prolongation result under depth and injectivity hypotheses, without
the quotient-space assumption.

\begin{mainthm}[Formal Prolongation Theorem]
\label{theorem:formal-prolongation}\normalfont
Let $(X,0)$ be a normal irreducible complex analytic germ of dimension
$n\geq3$, let $\omega$ be a holomorphic integrable 1-form on $(X,0)$ defining a codimension-one foliation $\F$ with
$\codim_X\Sing(\F)\geq2$, and let
$i:(Y,0)\hookrightarrow(X,0)$ be a generic normal surface section such that
$Y\cap(X_{\mathrm{sing}}\cup\Sing(\F))=\{0\}$ and $i^*\F$ has an isolated
singularity. Let
$f_0\in\OO_{Y,0}$ be a nonconstant holomorphic first integral of $i^*\F$.
 Put $I=\mathcal I_{Y,0}$ and set
\[
 \mathcal F_I^r\Omega^p_{X,0}
 :=I^r\Omega^p_{X,0}
   +d(I^{r+1})\wedge\Omega^{p-1}_{X,0}.
\]
Assume that, for every $k\geq1$,
\[
 \operatorname{depth}_{\mathfrak m_Y}(I^k/I^{k+1})\geq2,
\]
and that, for every $r\geq1$, the restriction homomorphism
\[
 \frac{\Omega^2_{X,0}}{\mathcal F_I^r\Omega^2_{X,0}}
 \longrightarrow
 H^0\!\left(Y\setminus\{0\},
 \left.
 \frac{\Omega^2_{X,0}}{\mathcal F_I^r\Omega^2_{X,0}}
 \right|_{Y\setminus\{0\}}
 \right)
\]
is injective. Equivalently, assume that
\[
 H^0_{\{0\}}\!\left(
 \frac{\Omega^2_{X,0}}{\mathcal F_I^r\Omega^2_{X,0}}
 \right)=0
 \qquad\text{for every }r\geq1.
\]

Then $f_0$ admits a nonconstant $I$-adic formal prolongation
\[
 \widehat f_I\in\varprojlim_k\OO_{X,0}/I^{k+1}
\]
satisfying $i^*\widehat f_I=f_0$ and
$\omega\wedge d\widehat f_I=0$. Its image in
$\widehat\OO_{X,0}$ is a formal first integral of $\F$. 
\end{mainthm}
This statement is
conditional and does not assert convergence. Here the displayed quotients are regarded as coherent sheaves on the
$r$-th infinitesimal neighborhood of $Y$, and their restriction is taken
to the corresponding punctured neighborhood over $Y\setminus\{0\}$.
This paper is organized as follows. In Section~\ref{sec:preliminaries}, we introduce the necessary background on complex analytic spaces, normal analytic varieties, holomorphic foliations, and resolution of singularities. Section~\ref{sec:holomorphic} is devoted to the existence of holomorphic first integrals, where we prove Theorems~\ref{thm:A} and~\ref{thm:B}, together with the Prolongation Theorem (Theorem~\ref{thm:E}) and its formal complement, Theorem~\ref{theorem:formal-prolongation}. Finally, in Section~\ref{sec:formal}, we study formal first integrals and prove Theorems~\ref{thm:C} and~\ref{thm:D}.

The problem of existence of holomorphic first integrals for foliations on singular spaces
has already been addressed, but from a different point of view. In \cite{CerveauLinsNeto2008}
the authors extend Malgrange's singular Frobenius theorem \cite{Malgrange1976} from a smooth ambient space to certain singular analytic varieties.

\section{Foliations and Resolution of Singularities on Analytic Spaces}\label{sec:preliminaries}

\subsection{Complex analytic subsets in \texorpdfstring{$\C^m$}{complex affine space}}
We begin by recalling some notions concerning analytic sets and varieties (\cite{Chirka1989}, \cite{GunningRossi1965}).
Let $\Omega\subset \C^m$ be an open subset. A subset $X\subset\Omega$ is a \textit{complex analytic subset} of $\Omega$ if for every point $p\in\Omega$ there is an open subset $p\in U_p\subset\Omega$ such that
\[
X\cap U_p=\{z\in U_p;\; f_1(z)=\cdots=f_r(z)=0\},
\]
for some holomorphic functions $f_1,\ldots,f_r\in\OO(U_p)$.

A point $p\in X$ is \textit{regular} if in a neighborhood $U$ of $p$ in $\Omega$ the intersection $X\cap U$ is a complex submanifold of $\Omega$. The non-regular points of $X$ are called \textit{singular points} of $X$.

We denote by $X_{\rm reg}\subset X$ the set of regular points of $X$. The set $X_{\rm reg}$ is an open subset of $X$, and the \textit{singular part} of $X$ is
\[
X_{\rm sing}:=X\setminus X_{\rm reg},
\]
which is an analytic subset of $\Omega$.

Given an open set $\Omega\subset\C^m$ and holomorphic functions $f_1,\ldots,f_r\in\OO(\Omega)$ we define the \textit{structure sheaf} of the analytic set
\[
X=\{z\in\Omega;\; f_1(z)=\cdots=f_r(z)=0\}
\]
as the quotient
\[
\OO_X:=\OO_\Omega/\J_X
\]
where $\OO_\Omega$ is the sheaf of holomorphic functions on $\Omega$ and $\J_X$ is the subsheaf of $\OO_\Omega$ consisting of the holomorphic functions that vanish identically on $X$.

\subsection{Complex analytic spaces}

A \textit{complex analytic space} consists of a pair $(X,\OO_X)$, where $X$ is a topological space and $\OO_X$ is a sheaf of local $\C$-algebras. It is subject to the following condition: given a point $p\in X$, there is a neighborhood $p\in U_p\subset X$ such that
\[
(X,\OO_X)|_{U_p}=(U_p,\OO_X|_{U_p})
\]
is isomorphic, as a locally ringed space, to a standard pair $(X_p,\OO_{X_p})$ for some analytic subset $X_p\subset\C^{m_p}$ and some $m_p\in\N$.

\subsection{Germs of complex analytic sets}

Let $X\subset\C^m$, $Y\subset\C^m$ be two analytic sets with $p\in X\cap Y\subset\C^m$. We say that $X$ and $Y$ define the same germ at $p$ if $X\cap W=Y\cap W$ for some open neighborhood $p\in W\subset\C^m$. This defines an equivalence class $(X,p)$ called the \textit{germ} of $X$ at $p$.

\subsection{Complex analytic surface germs}

By a complex analytic \textit{surface germ} we shall mean a germ $(X,p)$
of a complex analytic set such that $\dim_pX=2$.

\subsection{Resolution of singularities of analytic spaces}

Let $V$ be an analytic space. A \textit{resolution of the singularities} of $V$ consists of a proper analytic map $\pi\colon\widetilde V\to V$ of a complex (smooth) manifold $\widetilde V$ onto $V$, such that it defines a biholomorphism between the regular part $V_{\rm reg}$ and its inverse image $\pi^{-1}(V_{\rm reg})$, which is a dense subset of $\widetilde V$.

Plane curve singularities can always be resolved by quadratic blow-ups.

\begin{teo}[Theorem 1.1 in \cite{Laufer1971}]\normalfont
Let $N$ be a complex surface and $V\subset N$ a $1$-dimensional subvariety. There exists a complex surface $\widetilde N$ obtained from $N$ by successive quadratic blow-ups,
\[
\pi\colon\widetilde N\to N,
\]
such that
\[
\pi\colon\overline{\pi^{-1}(V_{\rm reg})}\to V
\]
is a resolution of the singularities of $V$. The map $\pi$ is locally given by a finite number of quadratic blow-ups.
\end{teo}

\subsection{Normal complex analytic sets}

Let $\Omega\subset\C^m$ be an open subset, $f_1,\ldots,f_r\in\OO(\Omega)$ and
\[
X=\{z\in\Omega;\; f_1(z)=\cdots=f_r(z)=0\}.
\]
The structural sheaf of $X$ is given by
\[
\OO_X=\OO_\Omega/\J_X
\]
where $\J_X\subset\OO_\Omega$ is the subsheaf of functions vanishing on $X$.

Given a point $p\in X$, we say that the local ring $\OO_{X,p}$ has no non-zero nilpotent elements if
\[
g_p\in\OO_{X,p},\;g_p^n=0\quad\Rightarrow g_p=0.
\]
In this case we say that $X$ is \textit{reduced at $p$}.

We say that $X$ is \textit{reduced} if it is reduced at every point $p\in X$.

We say that the ring $\OO_{X,p}$ is a \textit{normal ring} if it is an integral domain (it has no zero divisors) and it is integrally closed: if $h\in\Frac(\OO_{X,p})$ satisfies a monic equation
\[
h^\ell+a_1h^{\ell-1}+\cdots+a_\ell=0,\; a_j\in\OO_{X,p},
\]
then $h\in\OO_{X,p}$.

We then say that $(X,\OO_X)$ is a \textit{normal analytic variety} if it is reduced and $\OO_{X,p}$ is a normal ring for every $p\in X$.

\begin{exe}[normal hypersurfaces]
Let $X=\{f=0\}\subset \C^m$ be an analytic hypersurface germ, $f\in\OO_{\C^m,0}$ reduced germ. Then $X$ is normal if and only if the local ring
\[
\OO_{X,0}=\frac{\C\{z_1,\ldots,z_m\}}{(f)}
\]
is integrally closed.

If we assume that $X$ is reduced and irreducible we have:
$X $  is normal $\Longleftrightarrow$ $\codim_X(X_{\sing})\geq 2$, where
\[
X_{\sing}=\left\{z_0;\;
 f(z_0)=\frac{\partial f}{\partial z_1}(z_0)=\cdots=
 \frac{\partial f}{\partial z_m}(z_0)=0\right\}.
\]
A germ of hypersurface at $0\in\C^3$ is normal if the origin is an isolated singular point.
\end{exe}

\subsection{Resolution of singularities for normal surface germs}

Let $(X,p)$ be a germ of a normal complex analytic surface, with $p$ an isolated singularity. Then there exists a proper holomorphic map $\pi\colon\widetilde X\to X$ such that:
\begin{enumerate}
\item[(i)] $\widetilde X$ is a smooth complex surface.
\item[(ii)] $\pi$ induces a biholomorphism
\[
\pi\big|_{\widetilde X\setminus E}\colon\widetilde X\setminus E
\to X\setminus\{p\},
\]
where $E:=\pi^{-1}(p)\subset \widetilde X$
is the \textit{exceptional divisor}.

\item[(iii)] $E\subset \widetilde X$ is a compact analytic divisor, $E=\displaystyle\bigcup_{j=1}^{s}E_j$, where $E_j\subset\widetilde X$ is an irreducible compact complex curve.

\item[(iv)] $E=\displaystyle\bigcup_{j=1}^{s}E_j$ is a normal-crossings divisor:
\begin{itemize}
\item near a crossing point there are coordinates $(x,y)\in\widetilde X$ such that $E=\{xy=0\}$;
\item near a smooth point there are coordinates $(x,y)$ on $\widetilde X$ such that
$E=\{x=0\}$.
\end{itemize}

\item[(v)] The intersection matrix
$(E_i\cdot E_j)_{i,j=1,\ldots,r}$
is negative definite.
\end{enumerate}

We may assume that the resolution is \textit{minimal} in the following sense: $E$ has no component $E_j\simeq \Pone$ with $E_j\cdot E_j=-1$.

\begin{remark}About the resolution of singularities we have:
\begin{enumerate}
\item[(a)] The exceptional components need not be rational curves.

\item[(b)] The exceptional divisor is not necessarily a tree.
\end{enumerate}
\end{remark}

Let us expand these remarks.

\begin{exe}[rational singularities]
Let $(X,p)$ be a normal complex surface germ with a singularity at $p$. We follow \cite{Artin1966} and \cite{LeDuTrangTosun2004}.
 The singularity is called \textit{rational} if, for a sufficiently small representative $X_0$ of the germ, we have
\[
H^1(\widetilde X_0,\OO_{\widetilde X_0})=0
\]
in terms of the cohomology of holomorphic functions, where $\widetilde X_0\to X_0$ is a resolution of singularities.

If $(X,p)$ is a rational surface singularity then the exceptional divisor of the minimal resolution is a \textit{tree} of rational curves, in the following sense:

Let $E=\displaystyle\bigcup_{j=1}^{s}E_j$
be a normal-crossings divisor on a smooth complex surface $N$; each $E_j$
is an irreducible component of $E$.

The \textit{dual graph} of $E$ is the graph $\Gamma(E)$ obtained as follows:
\begin{itemize}
\item to each component $E_j$ we associate a vertex $v_j$;
\item if $E_i\cap E_j\neq\emptyset$ then we draw an edge between $v_i$ and $v_j$;
\item if $E_i$ and $E_j$ meet at more than one point, we draw one edge for each intersection point.
\end{itemize}

We shall say that $E$ is a \textit{tree} if $\Gamma(E)$ is a tree, i.e., $\Gamma(E)$ is connected and has no cycles. This is equivalent to saying that two components $E_i$ and $E_j$ are connected by a unique chain of components of $E$.

In short, $E$ is a tree if and only if the components of $E$ meet according to a connected graph without loops.
Examples of rational singularities are given below:

\begin{enumerate}
\item An example of rational normal singularity is
\[
X=\{z_1^2+z_2^2+z_3^{n+1}=0\}\subset (\C^3,0).
\]

\item Quotient surface singularities:
\[
(X,0)=(\C^2/G,0)
\]
where $G\subset GL(2,\C)$ is a finite group acting freely outside $0$. In this case $(X,0)$ is a normal surface singularity of rational type. Therefore, the exceptional divisor of the minimal resolution is a tree of rational curves.

\item As a particular case, consider
\[
(X,0)=(\C^2/G,0),
\]
where $G$ is the cyclic group generated by the matrix
\[
A=\begin{pmatrix}\xi&0\\0&\xi^q\end{pmatrix},
\]
where $\xi$ is a primitive $m$-th root of unity and
$\gcd(m,q)=1$.
\end{enumerate}

So the quotient $(X,0)$ has an isolated singularity at the origin, of normal rational type. The minimal resolution has exceptional divisor a linear chain $E_1-\cdots-E_s$ of projective lines $E_j$ with self-intersection $E_j^2=-b_j$ and the integers $b_j$ satisfy the Hirzebruch--Jung continued-fraction
\[
\frac{m}{q}=[b_1,\ldots,b_s].
\]

\end{exe}
\vglue.1in
Not all normal singularities exhibit a resolution whose exceptional divisor is a tree.
\begin{exe}
One example can be constructed using a cusp singularity and Grauert's
contraction theorem. Take a smooth complex surface $N$ and a
normal-crossings divisor $E=E_1+\cdots+E_r\subset N$, where $r\geq3$, each
$E_i\simeq\Pone$ and the dual graph is the cycle
$E_1-E_2-\cdots-E_r-E_1$. We may arrange that
$E_i^2=-3$, $E_i\cdot E_{i+1}=1$,
with indices taken cyclically, and that there are no other intersections.
The intersection matrix is negative definite, so Grauert's theorem ensures
that $E$ can be contracted to a point by a map $\pi\colon N\to X$. The
result is a normal complex analytic surface germ $(X,0)$, with
$0=\pi(E)$, having an isolated singularity at $0$ and an exceptional
divisor that is not a tree.
\end{exe}
The above discussion suggests the following definition:

\begin{defi}[normal-tree singularity]
A complex analytic surface germ $(X,0)$ is \textit{normal-tree} if it is
normal and admits a resolution $\pi:X^*\to X$ whose exceptional divisor
$E^*=\pi^{-1}(0)\subset X^*$ has a tree as its dual graph.
\end{defi}

\subsection{Order in the exceptional divisor}

Let $E=\displaystyle\bigcup_{j\in J}E_j$
be a normal-crossings divisor whose dual graph $\Gamma(E)$ is a tree. Choose
a \textit{root component} $E_0$. We may then define a partial order $\leq$
on the set of components of $E$ as follows:
\[
E_i\leq E_j \Longleftrightarrow \text{the unique path from }E_0\text{ to }E_j\text{ passes through }E_i.
\]
This order is not necessarily total, but is quite natural. If $E$ consists of finitely many components $E_j$ then $\Gamma(E)$ is a finite tree and therefore ``every chain has a maximal element''.

Thus, in the usual formulation of Zorn's lemma, we have ``every totally ordered subset has an upper bound'' (indeed, a maximum), and therefore we may apply Zorn's lemma to properties $P(E_j)$ compatible with the partial order introduced above.

\subsection{Holomorphic functions on complex analytic sets}

Let $X\subset\C^m$ be a complex analytic set with $0\in X$. By definition, a holomorphic function at $0\in X$ means a germ of function on $X$ at $0$, coming from a holomorphic function in the ambient space $\C^m$: $f\in\OO_{X,0}$ is represented by a function $f\colon X\cap\Omega\to\C$,
where $0\in\Omega\subset\C^m$ is an open set, such that there exists $F\in\OO_{\C^m,0}$ with
$f=F|_{X\cap\Omega}$ as germs, i.e.,
$f(x)=F(x)$, $\forall x\in X\cap\Omega'$ for some open sufficiently small neighborhood $0\in\Omega'\subset\C^m$.

If $F,G\in\OO_{\C^m,0}$ are such that $F$ and $G$ define the same germ $f\in\OO_{X,0}$, then $F\big|_X=G\big|_X$ in a neighborhood of $0$ in $X$ and therefore
\[
F-G\in\J_{X,0}=\{H\in\OO_{\C^m,0};\ H|_X=0\}.
\]
Therefore
\[
\OO_{X,0}=\OO_{\C^m,0}/\J_{X,0}.
\]

\subsection{Holomorphic \texorpdfstring{$1$-forms}{1-forms} on a complex analytic space}

Let $X$ be a complex analytic space. The \textit{sheaf of holomorphic K\"ahler differentials of $X$} is the sheaf $\Omega_X^1$ which is the $\OO_X$-module generated locally by symbols $df$, $f\in\OO_X$, satisfying the relations
\[
d(f+g)=df+dg,\;d(\lambda f)=\lambda df,\qquad \forall\lambda\in\C,
\]
and the Leibniz rule:
\[
d(fg)=f\,dg+g\,df.
\]

If $X\subset\C^m$ is an analytic set with ideal sheaf $\J_X\subset\OO_{\C^m}$ then
\[
\OO_X=\OO_{\C^m}/\J_X
\]
and
\[
\Omega_X^1\simeq
\frac{\Omega_{\C^m}^1\otimes_{\OO_{\C^m}}\OO_X}{d\J_X}
\simeq
\frac{\OO_X\,dz_1\oplus\cdots\oplus\OO_X\,dz_m}{\langle dg;\;g\in\J_X\rangle}.
\]

In particular, for a germ of analytic set
$(X,0)\subset(\C^m,0)$,
a holomorphic Kähler differential $\omega\in\Omega_{X,0}^1$ is represented, in coordinates $(z_1,\ldots,z_m)$, by
\[
\omega=\sum_{j=1}^{m}a_j\,dz_j,
\qquad a_j\in\OO_{\C^m,0},
\]
with the identifications
\[
a_j\sim \widetilde a_j \Longleftrightarrow a_j-\widetilde a_j\in\J_{X,0}
\]
and
\[
dg=0,
\qquad \forall g\in\J_{X,0}.
\]


The sheaf of germs of holomorphic $1$-forms on $X$ is the $\OO_{X,0}$-module $\Omega_{X,0}^1$ generated by the terms $df$, $f\in\OO_{X,0}$, with the rules
\[
d(f+g)=df+dg,
\qquad
 d(fg)=f\,dg+g\,df,
\qquad
 dc=0,
\quad \forall c\in\C.
\]
In the special case of an analytic germ $X\subset(\C^2,0)$, since
\[
\Omega_{\C^2,0}^1=\C\{x,y\}\,dx\oplus \C\{x,y\}\,dy,
\]
we obtain
\[
\Omega_{X,0}^1=\frac{\Omega_{\C^2,0}^1}{\J_{X,0}\Omega_{\C^2,0}^1+d\J_{X,0}}.
\]
An element $\omega\in\Omega_{X,0}^1$ is represented by
\[
\omega=A(x,y)\,dx+B(x,y)\,dy,
\qquad A,B\in\C\{x,y\}.
\]

\subsection{Holomorphic vector fields}

A \textit{holomorphic vector field on $(X,0)$} is a $\C$-linear derivation $D\colon\OO_{X,0}\to\OO_{X,0}$
satisfying the so-called Leibniz rule (product rule)
\[
D(fg)=fD(g)+gD(f).
\]
We denote the module of holomorphic vector fields by
\[
\Theta_{X,0}=\Der_\C(\OO_{X,0},\OO_{X,0})
\]
or also
\[
\Theta_{X,0}=\Hom_{\OO_{X,0}}(\Omega_{X,0}^1,\OO_{X,0})
\]
in an equivalent formulation.

In the special case $(X,0)\subset(\C^2,0)$, with coordinates $(x,y)$ as above, we can represent a vector field on $X$ by
\[
v=A(x,y)\frac{\partial}{\partial x}+B(x,y)\frac{\partial}{\partial y}
\]
in the ambient and tangent to $X$, i.e., $v(\J_{X,0})\subset\J_{X,0}$. Thus we have
\[
\Theta_{X,0}=\frac{\{v\in\Theta_{\C^2,0};\ v(\J_{X,0})\subset\J_{X,0}\}}{\J_{X,0}\cdot\Theta_{\C^2,0}}.
\]

\subsection{Holomorphic foliations on analytic varieties}

Let $X$ be a reduced complex analytic variety of dimension two. Denote by
\[
\Theta_X:=\Hom_{\OO_X}(\Omega_X^1,\OO_X)
\]
the tangent sheaf, i.e., the sheaf of holomorphic derivations of $\OO_X$.

\begin{defi}
A \textit{holomorphic foliation by curves of $X$} is a coherent subsheaf $\T_\F\subset\Theta_X$ with the following properties:
\begin{enumerate}
\item[(i)] $\T_\F$ has rank $1$ at a generic point;

\item[(ii)] $\T_\F$ is saturated in $\Theta_X$, i.e., $\Theta_X/\T_\F$ has no torsion;

\item[(iii)] $\T_\F$ is involutive: $[\T_\F,\T_\F]\subset\T_\F$. This condition is automatic on the regular part $X_{\mathrm{reg}}$.

\end{enumerate}
\end{defi}

The singular set of the foliation is the union of the set where $\T_\F$ fails to be locally free and the singular locus $X_{\sing}$ of $X$.

If $X$ is a normal $2$-variety, then it is regular in codimension one by Serre's criterion; consequently, $\codim_X X_{\sing}\geq 2$ \cite{Matsumura,Serre}. We then obtain:

If $X$ is a normal $2$-variety and $p\in X_{\sing}$, Lemma~1.1 of
\cite{Camacho1988} gives, for every holomorphic foliation by curves $\F$ on
$X$, a neighborhood $p\in U_p\subset X$ and a nonzero holomorphic vector
field $\vec F_p\in\Theta_X(U_p)$ tangent to $\F$. On the dense open set where
$\vec F_p$ does not vanish, its trajectories are contained in, and hence are
open subsets of, the leaves of $\F$. We do not claim that $\vec F_p$ freely
generates the possibly non-locally-free rank-one sheaf $\T_\F$ at $p$.

As a consequence:

\begin{propo}\normalfont
Let $X$ be a normal complex analytic $2$-variety.
\begin{enumerate}
\item[(i)] Any holomorphic foliation $\F$ (possibly singular) of dimension one on $X$ defines a nonsingular holomorphic $1$-dimensional foliation $\F_\mathrm{reg}$ (in the usual sense) on the set
\[
X_{\mathrm{reg}}\setminus \bigl(\Sing(\F)\cap X_{\mathrm{reg}}\bigr).
\]

\item[(ii)] Conversely, let $D\subset X_{\mathrm{reg}}$ be a discrete analytic subset, and let $\F^0$ be a nonsingular holomorphic foliation by curves on $X_{\mathrm{reg}}\setminus D$. Then $\F^0$ admits a unique extension to a saturated holomorphic foliation $\F$ on $X$. Moreover,
\[
\Sing(\F)\subset D\cup X_{\sing},
\qquad
\F|_{X_{\mathrm{reg}}\setminus D}=\F^0.
\]
\end{enumerate}
\end{propo}

\subsection{Holomorphic foliation germs and vector fields}

Let $(X,0)$ be a germ of a complex analytic variety. The ring $\OO_{X,0}$ is the local ring of germs of holomorphic functions, and $\mathfrak{m}_0\subset\OO_{X,0}$ is its unique maximal ideal.

Given a germ of a holomorphic foliation $\F$ on a normal surface germ
$(X,0)$, we choose, using Lemma~1.1 of \cite{Camacho1988}, a nonzero
holomorphic derivation $v\colon\OO_{X,0}\to\OO_{X,0}$ tangent to $\F$.
This means the following:

If $X\subset(\C^N,0)$ is defined by an ideal $\J_X\subset \C\{z_1,\ldots,z_N\}$ then
\[
\OO_{X,0}=\frac{\C\{z_1,\ldots,z_N\}}{\J_X}.
\]
A vector field $v\colon\OO_{X,0}\to\OO_{X,0}$ tangent to $X$ is represented by an ambient vector field
\[
\vec V\colon(\C^N,0)\to(\C^N,0),
\qquad
\vec V=\sum_{j=1}^{N}a_j(z)\frac{\partial}{\partial z_j},
\;
 a_j\in\OO_{\C^N,0},
\]
tangent to $X$, i.e., satisfying $\vec V(\J_X)\subset\J_X$. Thus we define $v\colon\OO_{X,0}\to\OO_{X,0}$ by setting
\[
v([f]):=[\vec V(f)],\qquad
\forall f\in\OO_{\C^N,0}=\C\{z_1,\ldots,z_N\}.
\]

Recall that, at a regular point $x\in X_{\reg}$, we have the tangent space
\[
T_xX:=T_x(X_{\reg})\subset T_x(\C^N)
\]
in the usual sense for smooth manifolds. We also have the isomorphism with Zariski tangent space
\[
T_xX\simeq \Hom_{\C}\left(\frac{\mathfrak{m}_x}{\mathfrak{m}_x^2},\C\right)
\simeq \left(\frac{\mathfrak{m}_x}{\mathfrak{m}_x^2}\right)^\vee .
\]

Given a foliation germ $\F$ on $(X,0)$, saying that $\F$ is generated by the derivation $v$ means that: for every $x\in X_{\reg}$ and away from the zeros of $v$, the tangent sheaf $T\F$ satisfies $T_x\F:=(T\F)_x$ is the complex line $\C v(x)\subset T_xX$.

\begin{propo}\normalfont
For a normal analytic surface germ $(X,0)$, a holomorphic foliation by curves
$\F$ admits a nonzero holomorphic derivation
$v\colon\OO_{X,0}\to\OO_{X,0}$ tangent to $\F$. Away from the zero set of
$v$ and the singular locus of $X$, the trajectories of $v$ are open subsets
of the leaves of $\F$.
\end{propo}

Let $(V,p)$ be a normal complex analytic surface germ with an isolated
singularity at $p$, and let $\pi\colon V^*\to V$ be a resolution. Then:
\begin{enumerate}
\item[(i)] $\pi\colon V^*\to V$ is a proper holomorphic map;

\item[(ii)] $\pi^{-1}(p)=\displaystyle\bigcup_{j=1}^{r}E_j$
is a union of compact Riemann surfaces $E_j$, with normal crossings;
\item[(iii)] $\pi\big|_{V^*\setminus\pi^{-1}(p)}\colon
V^*\setminus\pi^{-1}(p)\to V\setminus\{p\}$
is a holomorphic diffeomorphism.
\end{enumerate}

Let $\F$ be a holomorphic foliation in $(V,p)$, possibly singular at $p$.

\begin{lema}[Lemma 1.1 in \cite{Camacho1988}]
\label{lemma:camachovector}\normalfont
There is a vector field $\vec F\not\equiv0$, defined on a neighborhood $U$
of $p$ in $V$, whose trajectories in $U\setminus\{p\}$ are contained in
the leaves of $\F|_{U\setminus\{p\}}$.
\end{lema}
Lemma~\ref{lemma:camachovector} is also proved in \cite{GomezMont}.
\begin{remark}
The key point in the proof of the lemma above is the following.
\end{remark}

\begin{teo}[Levi extension theorem on normal spaces, \cite{GunningRossi1965}]\normalfont
Let $X$ be a normal complex analytic space and let $Y\subset X$ be a closed analytic subset of codimension $\codim_XY\geq2$. Then, every holomorphic function on $X\setminus Y$ extends to a holomorphic function on $X$. Equivalent form: if $j\colon X\setminus Y\hookrightarrow X$
denotes the inclusion, then $\OO_X\simeq j_*\OO_{X\setminus Y}$.
\end{teo}

\begin{cor}\normalfont
If $(X,p)$ is a normal complex analytic surface, then every holomorphic function on the punctured germ $X\setminus\{p\}$ extends to a holomorphic germ at $p$.
\end{cor}

Similarly, we have (\cite{Demailly} Definition 7.1, Theorem 7.3, and Definition 7.4, pp. 110–112):

\begin{teo}[Riemann extension theorem]
\label{theorem:riemmanextension}
\normalfont
Let $X$ be a normal complex analytic variety and let
$f\in\mathcal O(X_{\reg})$. If $f$ is locally bounded near
$X_{\sing}$, then there exists a unique function
$\widetilde f\in\mathcal O(X)$ such that
\[
\widetilde f|_{X_{\reg}}=f.
\]
\end{teo}

\subsection{Resolution of singularities for foliations on germs of normal surfaces}
The next proposition states that we may lift foliations on singular normal surfaces to
their desingularization.
\begin{propo}[\cite{Camacho1988}]\normalfont
Let $\F$ be a germ of a holomorphic foliation on a normal complex analytic
surface germ $(X,p)$. For a resolution $\pi_0\colon X^*\to X$ at $p$,
there exists a holomorphic foliation $\F^*$ on the smooth surface $X^*$
such that $\F^*$ coincides with the pullback $\pi_0^*\F$ on
$X^*\setminus\pi_0^{-1}(p)$ and
$\Sing(\F^*)\subset\pi_0^{-1}(p)$
is a finite set of isolated points.
\end{propo}

Now we can desingularize the foliation $\F^*$ using Seidenberg's method \cite{Seidenberg1968} via a finite composition of quadratic blow-ups, starting at singular points of $\F^*$ in $\pi^{-1}(p)$.

We finally obtain a proper holomorphic map
$\pi\colon\widetilde X\to X$, on a smooth complex surface $\widetilde X$, and a foliation $\widetilde\F$ on $\widetilde X$ such that:
\begin{enumerate}
\item[(i)] $\widetilde\F$ is regular on $\widetilde X\setminus\pi^{-1}(p)$ and coincides with $\pi^*(\F|_{X\setminus\{p\}})$;
\item[(ii)] $\Sing(\widetilde\F)\subset\pi^{-1}(p)$ consists of simple singularities;
\item[(iii)] $\pi^{-1}(p)=E$ is a normal-crossings divisor; $E$ can be written
$E=\displaystyle\bigcup_{i=1}^{r}K_i$, where each $K_i\subset\widetilde X$ is a compact Riemann surface;

\item[(iv)] $K_i$ is either invariant by $\widetilde\F$ or totally transverse to $\widetilde\F$;

\item[(v)] the resolution of the pair $((X,p),\F)$ is \textit{minimal} in the sense that if $\pi^{-1}(p)$ contains a projective curve $C\subset\pi^{-1}(p)$
with $C\cdot C=-1$,
then the blow-down of $C$ (via Castelnuovo Theorem \cite{Gunning1974}) results in a foliation with some non-reduced singularity.
\end{enumerate}

\begin{defi}
We shall refer to the map $\pi\colon\widetilde X\to X$
above and to the pair $(\widetilde X,\widetilde\F=\pi^*(\F))$ as a \textit{resolution of singularities of the pair $(X,\F)$}.
\end{defi}

We can also add the following to the above statement:

\begin{enumerate}
\item[(vi)] If $(X,0)$ is a normal-tree then the final exceptional divisor of the resolution of singularities of a pair $(X,\F)$ is a tree.
\end{enumerate}

Indeed, the resolution of singularities of a foliation does not introduce any cycles.

\begin{remark}
The condition that $(X,0)$ is a normal space is essential. Indeed, we still have a resolution of singularities when the space $(X,0)$ is reduced, not necessarily normal. However, the singular locus $X_{\sing}$ may have dimension $1$ and we can no longer assume that the exceptional divisor is a finite union of transverse compact analytic curves.
\end{remark}

\section{Holomorphic First Integrals of Foliations on Complex Analytic Spaces}\label{sec:holomorphic}

\subsection{Holomorphic First Integrals of Foliations on Analytic Spaces}
Let $X$ be a complex analytic $2$-variety and $\F$ a holomorphic foliation by curves on $X$.

A \textit{holomorphic first integral for $\F$} is a holomorphic function $f\colon X\to\C$ such that:
\begin{enumerate}
\item[(i)] $f$ is not constant;
\item[(ii)] $f$ is constant on each leaf of $\F$.
\end{enumerate}

If $p\in X$ is a point and $\F$ is a germ of foliation by curves at $p$ then by a (germ of) \textit{holomorphic first integral for $\F$} we mean a germ $f\in\OO_{X,p}$
admitting a representative
$f_U\colon U\subset X\to\C$ in a neighborhood $p\in U\subset X$, where $\F$ admits a representative $\F_U$ such that $f_U\colon U\to\C$ is a holomorphic first integral for $\F_U$ in $U$.

If this is the case then:
\begin{enumerate}
\item[(a)] the leaves of $\F_U$ are closed in $U\setminus\{p\}$;
\item[(b)] only finitely many leaves adhere to $p$.
\end{enumerate}

Or, equivalently:
\begin{enumerate}
\item[(a')] the leaves of $\F_U$ are either closed or separatrices;
\item[(b')] there are only finitely many separatrices.
\end{enumerate}

Indeed, under the assumption that all leaves are closed in $U\setminus\{p\}$, the Remmert--Stein theorem implies that the closure of every leaf adhering to $p$ is an invariant analytic curve, whose irreducible components are separatrices. Conversely, every separatrix determines a leaf adhering to $p$. Since an analytic curve germ has only finitely many irreducible components, the finiteness of the leaves adhering to $p$ is equivalent to the finiteness of the separatrices.

The classical theorem of Mattei--Moussu states that when $X$ is smooth (i.e., biholomorphic to a neighborhood of the origin $0\in\C^2$), the conditions above actually imply the existence of a holomorphic first integral. In dimension $n\geq2$ it reads as:

\begin{teo}[Theorem B page 473 in \cite{MatteiMoussu1980}]\normalfont
Let $\omega$ be a germ at $0\in\C^n$ of integrable holomorphic $1$-form. Then $\omega$ admits a holomorphic first integral if and only if there is a representative $\F_U$ for the foliation $\omega=0$ in  $U$ such that:
\begin{enumerate}
\item the leaves of the foliation $\F_U$ in $U$ are closed in $U\setminus\Sing(\omega)$;
\item only finitely many leaves of $\F_U$ adhere to $0$.
\end{enumerate}
\end{teo}

Our next result is a version of this theorem for normal singular varieties:

\begin{teo1}\normalfont
Let $(X,0)$ be a germ of a normal-tree complex analytic $2$-variety. Given a germ of a holomorphic foliation by curves $\F$ at $0$ in $(X,0)$, the following conditions are equivalent:
\begin{enumerate}
\item[(i)] $\F$ admits a holomorphic first integral;
\item[(ii)] $\F$ has a representative $\F_U$ in a neighborhood
$0\in U\subset X$ such that:
\begin{enumerate}
\item[(ii.a)] $\F_U$ has closed leaves in $U\setminus\{0\}$;
\item[(ii.b)] only finitely many leaves adhere to the origin $0\in X$.
\end{enumerate}
\end{enumerate}
\end{teo1}

Theorem \ref{thm:A} will be proved throughout the next sections.

\subsection{Holomorphic foliations, closed leaves, separatrices and first integral}

Let $\F$ be a holomorphic foliation on a germ of a normal complex analytic $2$-variety $(X,0)$. We consider a resolution $\pi\colon\widetilde X\to X$
of the pair $((X,0),\F)$, with $\pi^*(\F)=\widetilde\F$
and $E=\displaystyle\bigcup_{j\in J}E_j=\pi^{-1}(0)$.

We say that a germ of irreducible analytic curve $S\subset(X,0)$ through $0$ is a \textit{separatrix of $\F$} if $S$ is $\F$-invariant, i.e., $S\setminus\{0\}$ is a leaf of $\F\big|_{X\setminus\{0\}}$.

The germ $\F$ is \textit{dicritical} if it admits infinitely many separatrices and \textit{non-dicritical} otherwise.

\begin{lema}\normalfont
The germ of foliation $\F$ is non-dicritical if and only if every component $E_j$ of the exceptional divisor $\pi^{-1}(0)$ is invariant by $\widetilde\F$.
\end{lema}
\begin{proof} This is standard for foliation germs at $0\in \mathbb C^2$.
The argument for germs on singular normal varieties is the same.
If the exceptional divisor of the resolution of the pair $((X,0),\F)$ exhibits some noninvariant
component, then this component is intersected transversely by infinitely many leaves that
project by the resolution morphism onto separatrices of the original foliation. In this case the foliation is dicritical. Conversely,
if all the components of the exceptional divisor of the resolution of
$((X,0),\F)$ are invariant, then the only separatrices of the original
foliation are those that lift to separatrices of singularities of the
lifted foliation $\widetilde\F$. Since $\widetilde\F$ has only finitely
many singularities and each reduced singularity has at most two
separatrices, the foliation is non-dicritical.

\end{proof}

Next, we shall need the following lemma.
\begin{lema}\normalfont
For a leaf $L$ of $\F$, the following conditions are equivalent:
\begin{enumerate}
\item[(i)] $L$ is closed in $U$ (and $0\notin\overline{L}$);
\item[(ii)] $\widetilde L$ is closed in $\widetilde U$ and $\overline{\widetilde L}\cap E=\emptyset$.
\end{enumerate}
\end{lema}

\begin{proof}\textup{(i)$\Rightarrow$(ii).} Assume that $L$ is closed in
$U$ and hence that $0\notin\overline{L}$. Then there is an open
neighborhood $0\in A\Subset U$ such that
$L\cap\overline A=\emptyset$; see Figure~\ref{fig:leaf-away-origin}.
\begin{figure}[ht!]
    \centering
    \includegraphics[width=0.5\linewidth]{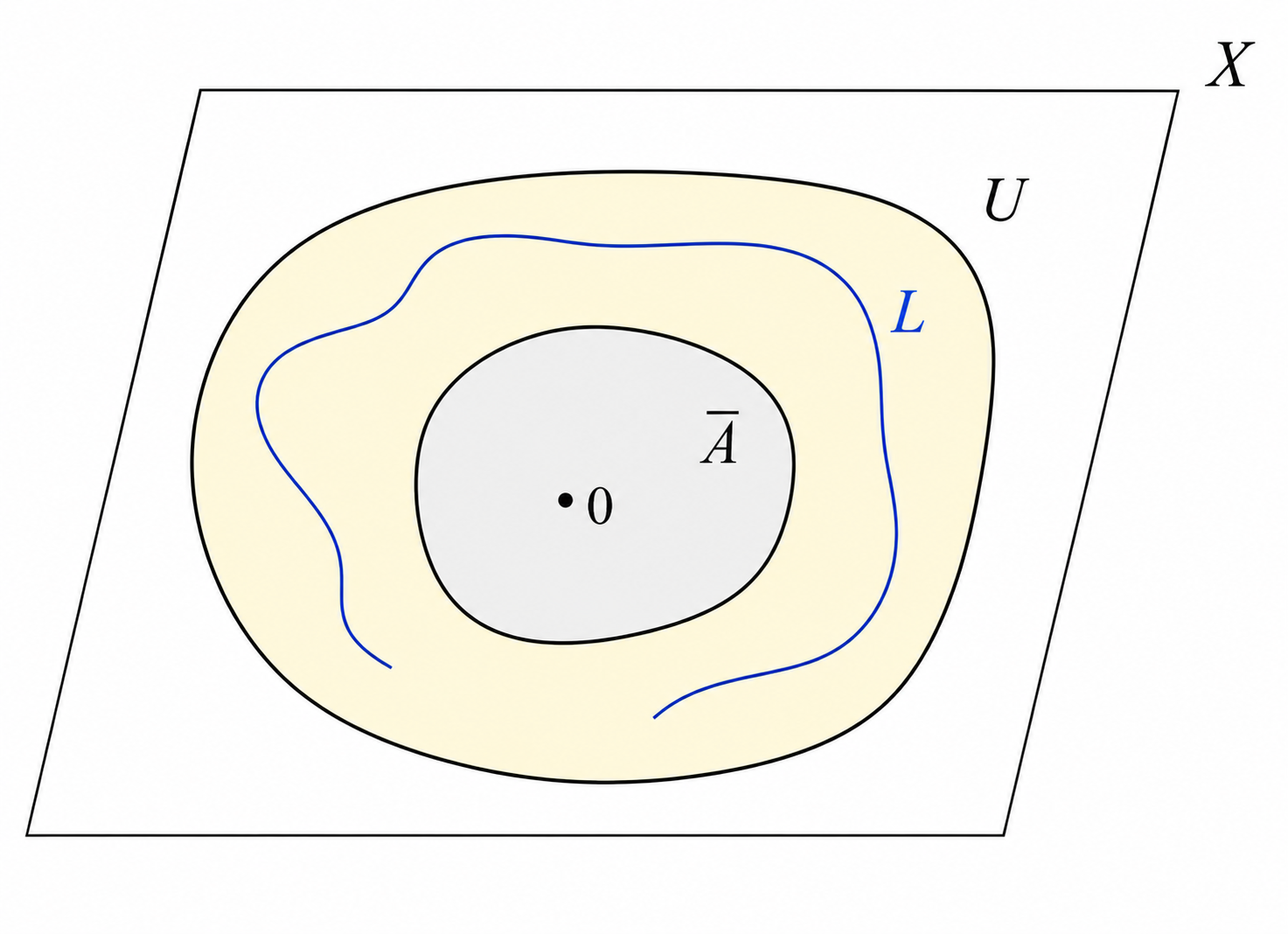}
    \caption{A relatively compact neighborhood of the origin disjoint from
    the closed leaf $L$.}
    \label{fig:leaf-away-origin}
\end{figure}

The inverse image $\pi^{-1}(A)=:\widetilde A$
is an open neighborhood of $\pi^{-1}(0)=E$ in $\widetilde X$. Since $L\cap A=\emptyset$ we have $\pi^{-1}(L)\cap\pi^{-1}(A)=\emptyset$ (recall that
$\widetilde p\in\pi^{-1}(L)\cap\pi^{-1}(A)
\Rightarrow
\pi(\widetilde p)\in L\cap A$). Thus $\widetilde L\cap\widetilde A=\emptyset$
and $\widetilde A$ is an open neighborhood of $E=\pi^{-1}(0)$ in $\widetilde X$.
\begin{figure}[ht!]
    \centering
    \includegraphics[width=0.5\linewidth]{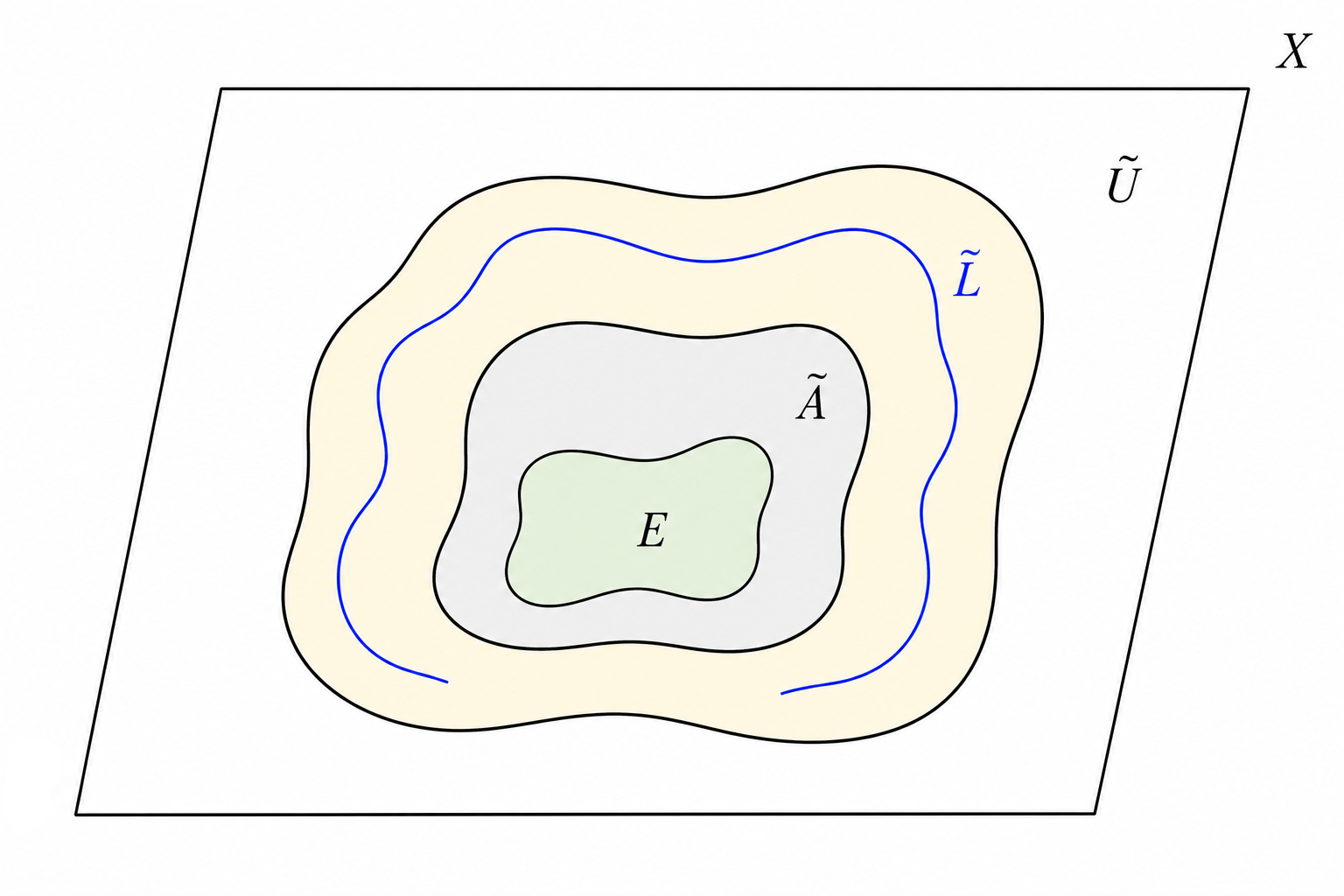}
    \caption{The lifted leaf $\widetilde L$ and a neighborhood
    $\widetilde A$ of the exceptional divisor $E$.}
    \label{fig:lifted-leaf}
\end{figure}
Therefore $\overline{\widetilde L}\cap E=\emptyset$. Because
\[
\pi\big|_{\widetilde U\setminus E}\colon
\widetilde U\setminus E\to U\setminus\{0\}
\]
is a homeomorphism, we conclude that $\widetilde L$ is closed in $\widetilde U$.

\noindent $(ii)\Rightarrow(i)$: Assume now that $\widetilde L$ is closed in $\widetilde U$ and that $\widetilde L\cap E=\emptyset$. Then there is an open neighborhood $\widetilde A\subset\widetilde U$ of $E$ such that $\widetilde A\cap\widetilde L=\emptyset$. Since $\pi$ is proper, it is a closed map. Hence
\[
A:=U\setminus\pi(\widetilde U\setminus\widetilde A)
\]
is an open neighborhood of $0$, because $\pi^{-1}(0)=E\subset\widetilde A$. Moreover, $\pi^{-1}(A)\subset\widetilde A$, and therefore $L\cap A=\emptyset$. Thus, $0\notin\overline L$. Finally, since $\widetilde L$ is closed and $\pi\colon\widetilde U\to U$ is proper, $L=\pi(\widetilde L)$ is closed in $U$.
\end{proof}
In the above framework, by a leaf of $\F$ we mean a leaf $L\subset X\setminus\{0\}$
of the nonsingular foliation $\F\big|_{X\setminus\{0\}}$. Let $0\in U\subset X$ be a small neighborhood of $0$ in $X$, a representative $\F_U$ for $\F$ in $U$, and let
$L\subset U\setminus\{0\}$ be a leaf of $\F_U$.

\begin{lema}
\label{lemma:4}\normalfont Assume that $\F$ is non-dicritical. The following conditions are equivalent:
\begin{enumerate}
\item[(i)] $L$ is closed in $U$;
\item[(ii)] $\pi^{-1}(L)$ is closed in $\widetilde U:=\pi^{-1}(U)\subset\widetilde X$.
\end{enumerate}
Moreover:
\begin{enumerate}
\item[(a)] $0\in\overline{L}\subset U$ iff $\overline{\pi^{-1}(L)}\cap E\neq\emptyset$ in $\widetilde U$;

\item[(b)] If $\overline L$ is a separatrix of $\F$, then $L$ is
closed in $U\setminus\{0\}$ and $0\in\overline L$. Conversely, if
$L$ is closed in $U\setminus\{0\}$ and $0\in\overline L$, then
$\overline L=L\cup\{0\}$ is a one-dimensional invariant analytic
subset of $U$. Consequently, every irreducible component of the germ
$(\overline L,0)$ is a separatrix of $\F$. In particular, if
$(\overline L,0)$ is irreducible, then $\overline L$ is a separatrix.
\end{enumerate}
\end{lema}

\begin{proof}[Proof of Lemma ~\ref{lemma:4}]
(i)$\Rightarrow$(ii): Since $\pi\colon\widetilde X\to X$ is continuous, if $L\subset U$ is closed then
$\widetilde L=\pi^{-1}(L)\subset \widetilde X$ is closed.

(ii)$\Rightarrow$(i). Let $\widetilde L\subset\widetilde X$ be closed. Since $\pi\colon\widetilde X\to X$ is a proper map, it is closed and therefore $L=\pi(\widetilde L)$ is closed in $U$.

(a) Since $\pi\big|_{\widetilde X\setminus E}\colon\widetilde X\setminus E\to X\setminus\{0\}$ is a homeomorphism, we conclude that
$0\in\overline{L}$ iff $\pi^{-1}(0)\cap\overline{\widetilde L}\neq\emptyset$, where $\widetilde L=\pi^{-1}(L)$.

(b) Let $\overline{L}$ be a separatrix of $\F$. Then $\overline{L}\setminus\{0\}$ is a leaf of $\F$ and therefore
$\overline{L}\setminus\{0\}=L$,
so that $L$ is closed in $U\setminus\{0\}$ and $0\in\overline{L}\setminus L\subset \overline{L}$.
Assume now that $L$ is closed in $U\setminus\{0\}$ and that
$0\in\overline L$. Since $L$ is a leaf of the regular holomorphic
foliation on $U\setminus\{0\}$ and is closed there, it is a
one-dimensional analytic subset of $U\setminus\{0\}$. Moreover,
\[
\dim L=1>\dim\{0\}=0.
\]
The Remmert--Stein extension theorem therefore implies that
\[
\overline L=L\cup\{0\}
\]
is a one-dimensional analytic subset of $U$. It is invariant under
$\F$ because its regular part is a leaf. Hence every irreducible
component of the germ $(\overline L,0)$ is a separatrix of $\F$.
If this germ is irreducible, then $\overline L$ itself is a
separatrix.
\end{proof}
We denote by $\operatorname{Diff}(\mathbb C,0)$ the group of germs at the
origin of holomorphic diffeomorphisms fixing the origin. An element
$h\in\operatorname{Diff}(\mathbb C,0)$ is \emph{periodic} if
$h^k=\operatorname{id}$ for some integer $k\geq1$. An important ingredient
in the Mattei--Moussu theorem is the following topological criterion for
periodicity.

\begin{defi}
    \label{Definition:finiteorbits}
Let $h\in\operatorname{Diff}(\mathbb C,0)$ be represented by
$h:U\to h(U)$, and let $V\subset U$. The $V$-pseudo-orbit of $x\in V$ is
\begin{align*}
\mathcal O_{V,h}(x)
={}&\{x\}\\
&\cup\{h^k(x):h^j(x)\in V\text{ for }j=1,\ldots,k\}\\
&\cup\{h^{-k}(x):h^{-j}(x)\in V\text{ for }j=1,\ldots,k\}.
\end{align*}

We say that $h$ has {\it finite pseudo-orbits} if there are neighborhoods
$0\in V\subset U$ such that $h$ and $h^{-1}$ are defined on $U$ and
$\#\mathcal O_{V,h}(x)<\infty$ for every $x\in V$.
\end{defi}

Th\'eor\`eme~2 on page~477 of \cite{MatteiMoussu1980} states that an
element $h\in\operatorname{Diff}(\mathbb C,0)$ is periodic if and only if
its pseudo-orbits are finite. The same argument gives the following
slightly more flexible form.

\begin{propo}
\label{Proposition:finiteorbits}
 Suppose that there is a
countable subset $A$ of a sufficiently small  disk $D(0;r)\subset \mathbb C$ such that every
point outside $A$ has finite pseudo-orbit under $h$. Then $h$ is periodic.
\end{propo}
Using this we obtain, for finitely generated groups of germs:

\begin{lema}\label{lemma:finite-holonomy}\normalfont Let $G<\operatorname{Diff}(\mathbb C, 0)$ be a finitely generated subgroup of germs of holomorphic diffeomorphisms.
 Suppose that there is a
countable subset $A$ of a sufficiently small  disk $D(0;r)\subset \mathbb C$ such that every
point outside $A$ has finite pseudo-orbit under $G$.
 Then $G$ is finite cyclic and analytically conjugate to
 $\langle z\mapsto e^{2\pi i/\nu}z\rangle$ for some $\nu\geq1$.
\end{lema}

\begin{proof} There are maps $f_j$, $j=1,\ldots,r$, in $G$ such that
$G=\langle f_1,\ldots,f_r\rangle$.
\begin{claim} $G$ is abelian.
\end{claim}

Indeed, given two maps $h_{1},h_{2}\in G$, the commutator
$f=[h_{1},h_{2}]=h_{1}h_{2}h_{1}^{-1}h_{2}^{-1}$ belongs to $G$ and is
tangent to the identity: $f'(0)=1$. If $f\neq\operatorname{id}$, then
$f(z)=z+a_{k+1}z^{k+1}+\cdots$ for some $a_{k+1}\neq0$ and $k\geq1$.
This contradicts the fact that $f$ must have finite pseudo-orbits
(\cite{Camacho1978}). Therefore $f=\operatorname{id}$ and
$h_{1}h_{2}=h_{2}h_{1}$, i.e., $h_{1}$ and $h_{2}$ commute.

By Proposition~\ref{Proposition:finiteorbits}, each generator $f_j$ is
periodic. Since $G=\langle f_1,\ldots,f_r\rangle$ is abelian, it follows
that $G$ is finite.

It remains to linearize the action. Define
\[
\Phi(z):=\sum_{g\in G}\frac{g(z)}{g'(0)}.
\]
Then $\Phi(0)=0$ and $\Phi'(0)=|G|\neq0$, so
$\Phi\in\operatorname{Diff}(\mathbb C,0)$. For $g_0\in G$,
\[
\begin{aligned}
\Phi(g_0(z))
&=\sum_{g\in G}\frac{(g\circ g_0)(z)}{g'(0)}\\
&=g_0'(0)\sum_{g\in G}
  \frac{(g\circ g_0)(z)}{(g\circ g_0)'(0)}\\
&=g_0'(0)\Phi(z).
\end{aligned}
\]
Therefore
\[
\Phi\circ g_0\circ\Phi^{-1}(w)=g_0'(0)w.
\]
Thus $G$ is conjugate to a finite subgroup of $\mathbb C^*$, and every
finite subgroup of $\mathbb C^*$ is cyclic.
\end{proof}

\begin{lema}[Gluing along a compact exceptional tree]
\label{lemma:tree-gluing}\normalfont
Let $D=\bigcup_{j=1}^rD_j$ be a connected normal-crossings divisor in a
smooth complex surface, with compact irreducible components and finite tree
dual graph. Let $\mathcal G$ be a reduced holomorphic foliation leaving
$D$ invariant. Assume that every singular point of $\mathcal G$ on $D$
admits a nonconstant holomorphic first integral.

For each component put
$D_j^\circ=D_j\setminus\Sing(\mathcal G)$, choose a finite generating
system of $\pi_1(D_j^\circ)$, and transport to one fixed transversal to
$D_j$ both the corresponding holonomy maps and the invariance groups of
the local first integrals at the punctures. If the subgroup
$K_j<\Diff(\C,0)$ generated by all these maps is finite for every $j$,
then $\mathcal G$ admits a nonconstant holomorphic first integral in a
neighborhood of $D$.

Suppose, alternatively, that $\mathcal G$ admits a nonconstant formal first
integral $\widehat H$ along $D$, nonconstant on a generic transversal to
each component, and that at every singular point $p$ there are a primitive
holomorphic first integral $Q_p$ and a formal one-variable series
$\widehat\psi_p$ such that
\[
 \widehat H_p=\widehat\psi_p\circ Q_p.
\]
Then the holomorphic first integral may be chosen so that, near every
connected subtree on which it has been constructed,
\[
 \widehat H=\widehat\Phi\circ F
\]
for a nonconstant formal one-variable series $\widehat\Phi$.
\end{lema}

\begin{proof}
For one component $D_j$, choose a coordinate $z$ on the fixed transversal
which linearizes the finite group $K_j$ and put $P_j(z)=z^{|K_j|}$.
Invariance under the holonomy generators makes the usual holonomy
continuation of $P_j$ single-valued along $D_j^\circ$. At a puncture $p$,
let $Q_p$ be a primitive local first integral. Since $K_j$ contains the
transported invariance group of $Q_p|_\Sigma$, the argument of
\cite[Chapter~V, \S1, Lemmas~2 and~3]{MatteiMoussu1980} gives
\[
 P_j|_\Sigma=\varphi_p\circ Q_p|_\Sigma
\]
for a holomorphic one-variable germ $\varphi_p$. Thus
$\varphi_p\circ Q_p$ extends the continued function across $p$. This proves
the assertion near one component. Notice that the handle generators of a
positive-genus component have been included explicitly; this is the only
addition to the proof for an exceptional $\mathbf P^1$.

Root the dual tree. Suppose an integral has been constructed near a
connected rooted subtree and attach an adjacent component $D_j$ at its
unique corner $p$. Let $Q_p$ be a primitive local first integral at $p$.
After changing coordinates on the two transversals, its restrictions are
monomials. The quotient coordinates obtained from the finite groups on the
old and new sides therefore extend at the corner as
\[
 Q_p^{a}\qquad\text{and}\qquad Q_p^{b}
\]
for some positive integers $a,b$. Put $c=\operatorname{lcm}(a,b)$.
Replacing the old integral by its $c/a$-th power and the new integral by
its $c/b$-th power makes both equal to $Q_p^c$ near $p$. These replacements
are nonconstant one-variable compositions, so they remain first integrals.
They also preserve all equalities already established on the rooted
subtree, because the same composition is applied to the single integral
defined there. The two functions now glue across the corner. Induction over
the finite tree gives a nonconstant holomorphic first integral near all of
$D$.

For the formal assertion, let $\widehat h$ be the restriction of
$\widehat H$ to a generic transversal to the component currently being
attached and put
\[
 G(\widehat h)=
 \{g\in\Diff(\C,0):\widehat h\circ g=\widehat h\}.
\]
Lemma~\ref{lemma:6} shows that this analytic invariance group is finite.
After linearizing it, write its order as $m$ and take the quotient
coordinate $P(z)=z^m$. The one-variable formal invariance theorem
\cite[Chapter~I, Proposition~1.2]{MatteiMoussu1980} gives a formal series
$\widehat\Phi$ such that
\[
 \widehat h=\widehat\Phi\circ P.
\]
All component holonomies and all transported puncture groups preserve
$\widehat h$, so they lie in $G(\widehat h)$ and $P$ extends over the
component by the preceding construction.

At a corner $p$, write the primitive local integral as $Q_p$ and use the
hypothesis
\[
\widehat H_p=\widehat\psi_p\circ Q_p.
\]
Apply Lemma~\ref{lemma:6} to the one-variable formal germ
$\widehat\psi_p$ and let $R_p$ be the holomorphic quotient coordinate of
its finite analytic invariance group. The one-variable formal invariance
theorem gives
\[
 \widehat\psi_p=\widehat\chi_p\circ R_p
\]
for a formal series $\widehat\chi_p$. Consequently,
\[
 F_p:=R_p\circ Q_p
\]
is a holomorphic first integral defined on the whole corner chart and
$\widehat H_p=\widehat\chi_p\circ F_p$. On either transverse side of the
corner, $F_p$ is the maximal holomorphic quotient through which the
corresponding restriction of $\widehat H$ factors. The quotient coordinate
constructed along the adjacent component has the same maximality property;
the one-variable uniqueness statement
\cite[Chapter~I, Proposition~1.2]{MatteiMoussu1980} therefore identifies it
with $F_p$ up to an invertible holomorphic reparametrization. Choosing this
reparametrization makes the quotient coordinates on the two sides agree
with $F_p$. Hence
\[
\widehat H=\widehat\Phi\circ F
\]
continues to hold on the enlarged neighborhood. Because a new component
meets the previously constructed divisor at a unique corner, no additional
compatibility condition arises. Induction over the tree proves the formal
assertion.
\end{proof}

\subsection{Proof of Theorem A}

We first prove \textup{(i)$\Rightarrow$(ii)}. Let
$f\in\OO_{X,0}$ be a nonconstant holomorphic first integral and choose a
sufficiently small representative. Every leaf in $U\setminus\{0\}$ is a
connected component of the regular part of a fiber of $f$, and is therefore
closed in $U\setminus\{0\}$. A leaf adhering to the origin is contained in
$f^{-1}(f(0))$. This analytic curve germ has finitely many irreducible
components, and the regular part of each component has finitely many
connected components in a sufficiently small representative. Hence only
finitely many leaves adhere to $0$.

We now prove \textup{(ii)$\Rightarrow$(i)}.
Assume condition \textup{(ii)} of Theorem~A. Since the leaves are closed in $U\setminus\{0\}$ and only finitely many leaves adhere to $0$, the preceding observation shows that $\F$ has only finitely many separatrices. Hence $\F$ is non-dicritical.

Let $E_j\subset E$ be a component of the exceptional divisor $E=\pi^{-1}(0)$, where
$\pi\colon\widetilde X\to X$ is a minimal resolution of the pair $((X,0),\F)$. By the preceding lemma, every component $E_j$ of $E=\displaystyle\bigcup_{j=1}^{r}E_j$ is invariant by the lifted foliation $\widetilde\F=\pi^*(\F)$. We fix a regular point $q_j\in E_j\setminus\Sing(\widetilde\F)$,
which is not a corner, and a transverse disk $\Sigma_j\transv E_j$ centered at $q_j$.
\begin{figure}[ht!]
    \centering
    \includegraphics[width=0.5\linewidth]{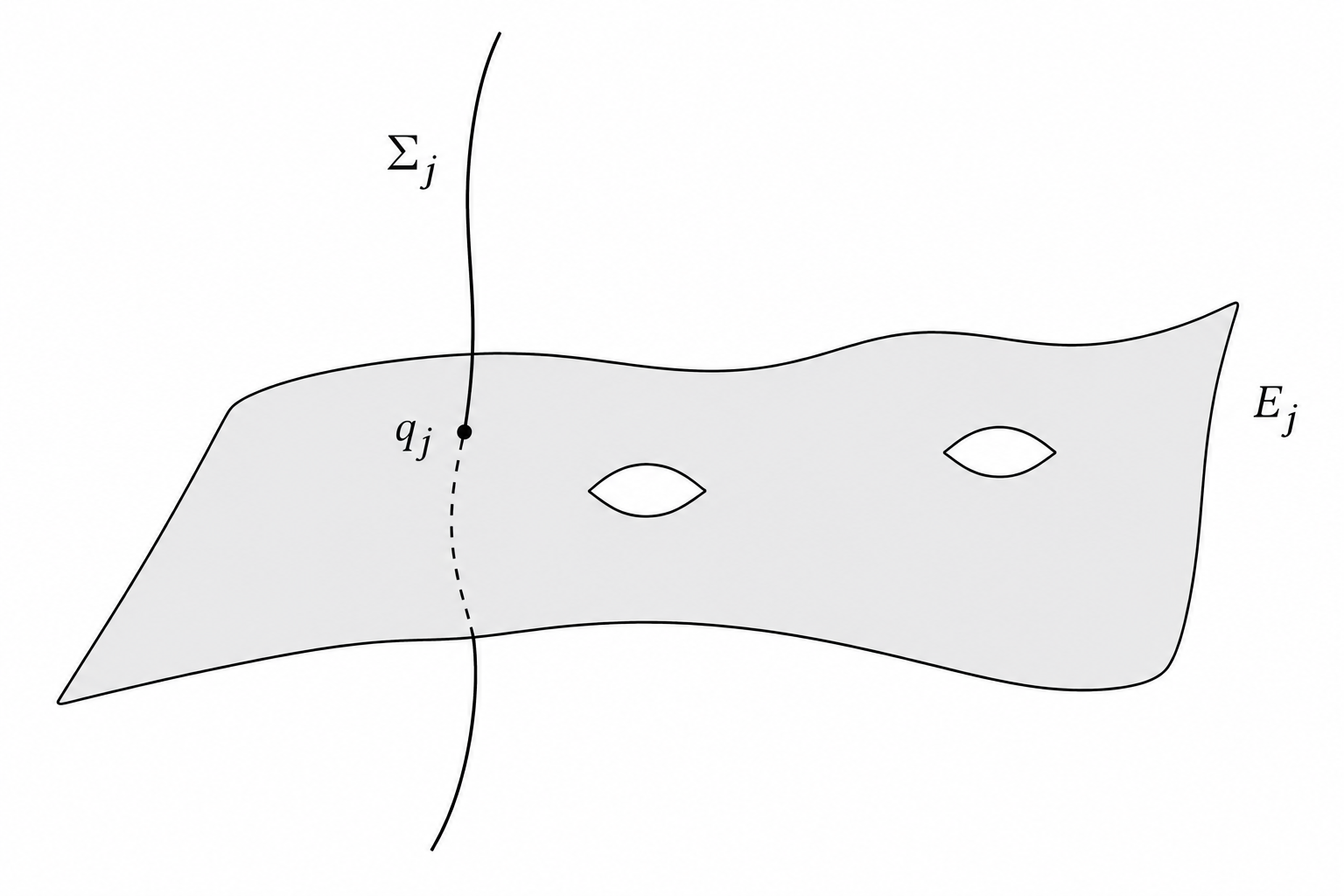}
    \caption{A transverse disk $\Sigma_j$ through a regular point $q_j$ of
    the exceptional component $E_j$.}
    \label{fig:transversal-Ej}
\end{figure}
We consider the holonomy group
\[
H(E_j):=\Hol(\widetilde\F,L_j,\Sigma_j,q_j)
\subset\Diff(\Sigma_j,q_j)\simeq\Diff(\C,0)
\]
of the leaf
\[
L_j:=E_j\setminus\bigl(E_j\cap\Sing(\widetilde\F)\bigr).
\]

By condition \textup{(ii)} of Theorem~A, the leaves of $\F$ are closed in
$U\setminus\{0\}$, and only finitely many leaves adhere to $0$.

Let $\mathscr A$ denote the finite family of leaves that adhere to $0$. Choose
\[
x\in\Sigma_j\setminus\{q_j\}
\]
such that the corresponding leaf $L_x$ does not belong to $\mathscr A$. Then
\[
0\notin\overline{L_x}.
\]
Consequently, if $\widetilde L_x$ denotes the lift of $L_x$ to $\widetilde X$, then
\[
\overline{\widetilde L_x}\cap E=\varnothing.
\]
After replacing $\Sigma_j$ by a relatively compact transverse disk, the intersection $\widetilde L_x\cap\Sigma_j$ is finite. Hence the pseudo-orbit of $x$ under the holonomy group $H(E_j)$ is finite.
The intersections with $\Sigma_j$ of the finitely many leaves in
$\mathscr A$ form a countable subset. For every point outside this subset,
the preceding argument gives a finite pseudo-orbit. The leaf
$L_j=E_j\setminus\Sing(\widetilde\F)$ is a compact Riemann surface with
finitely many punctures, so its fundamental group, and hence its holonomy
group $H(E_j)$, is finitely generated. Lemma~\ref{lemma:finite-holonomy}
therefore shows that $H(E_j)$ is finite cyclic and analytically linearizable
as a group of rotations:
\[
H(E_j)\simeq\left\langle z\mapsto e^{\frac{2\pi i}{\nu_j}}z\right\rangle
\]
for some $\nu_j\in\N$.

\begin{lema}
\label{Lemma:5}\normalfont
Assume condition \textup{(ii)} of Theorem~A. Then $\F$ is
non-dicritical, as proved above. With the preceding notation, there is a
neighborhood $A_j$ of $L_j:=E_j\setminus\bigl(E_j\cap\Sing(\widetilde\F)\bigr)$
in $\widetilde X$ where $\widetilde\F$ admits a holomorphic first integral
$F_j\colon A_j\to\C$. Moreover, because the dual graph of $E$ is a finite
tree, these local first integrals may be reparametrized and glued to a
nonconstant holomorphic first integral $F$ in a neighborhood of the entire
exceptional divisor.
\end{lema}
\begin{figure}[ht!]
    \centering
    \includegraphics[width=0.5\linewidth]{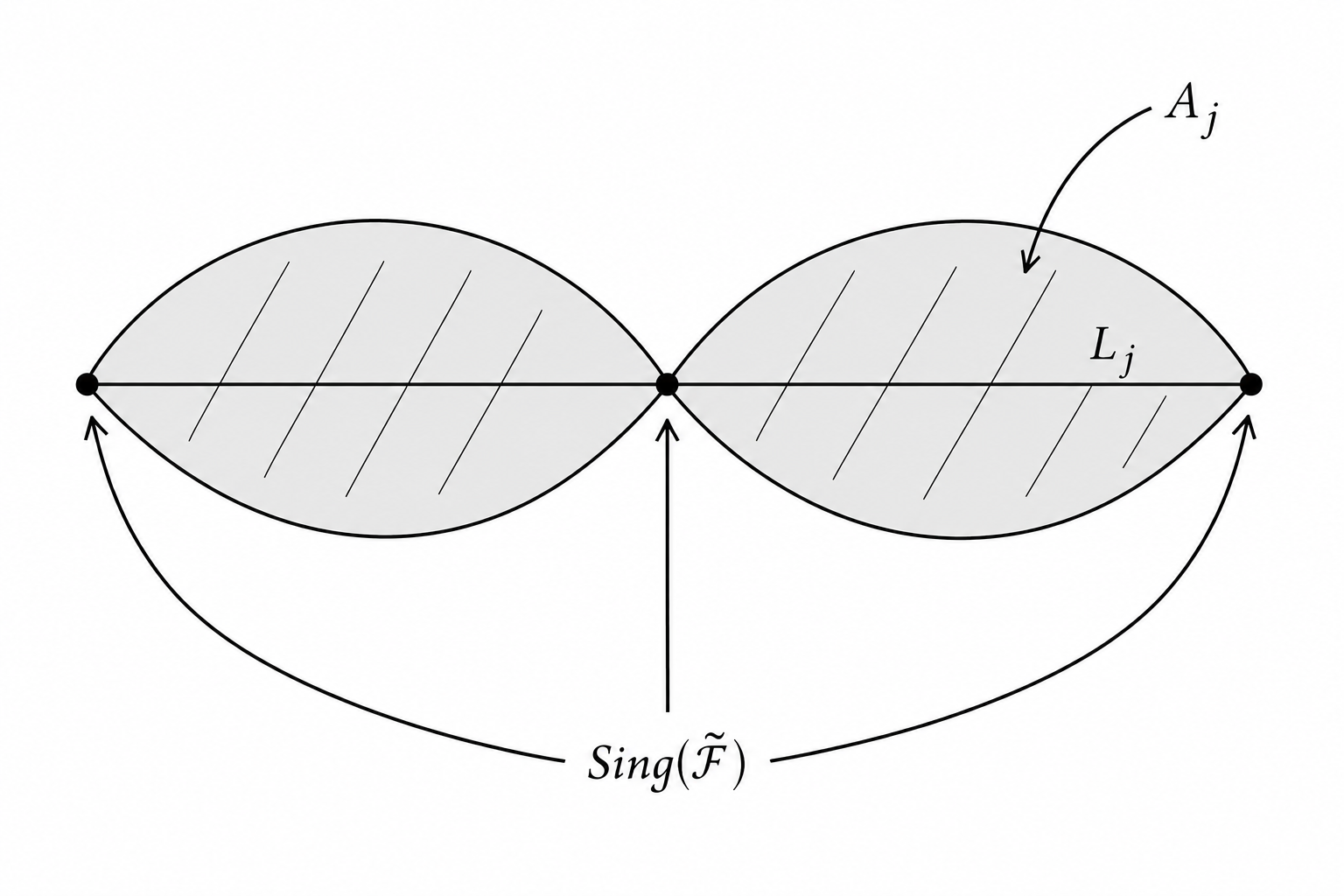}
    \caption{A neighborhood $A_j$ of
    $L_j=E_j\setminus\Sing(\widetilde\F)$, covered by local flow boxes.}
    \label{fig:neighborhood-Lj}
\end{figure}

\begin{proof}
We first determine the singularities on $E$.

\begin{claim}\label{claim1}\normalfont
$p_j$ is not a saddle-node.
\end{claim}
\begin{proof}[Proof of the Claim 1]
Suppose that $p_j$ is a saddle-node. Denote by $S(p_j)$ its strong manifold.
The local holonomy map of this separatrix is a non-trivial map tangent to the identity of the form
\[
h(z)=z+a_{k+1}z^{k+1}+\cdots,
\qquad a_{k+1}\neq0.
\]
If $S(p_j)\subset E_i$ for some $i$, then, arguing as we did for $E_j$, we conclude that $H(E_i)$ is finite and therefore $h$ must be periodic, a contradiction.
If $S(p_j)\not\subset E$, then $S(p_j)$ is the strict transform of a
separatrix through the origin.  The holonomy of the strong separatrix of a
saddle-node is a nonperiodic parabolic germ.  By the Leau--Fatou flower
theorem, every sufficiently small punctured transversal contains an open
petal on which an infinite forward or backward orbit remains defined and
converges to the fixed point. As the initial point varies in the petal,
this produces infinitely many distinct leaves accumulating on $E$ and,
after blowing down, adhering to the origin. This contradicts
condition~\textup{(ii.b)}. Hence this case is
also impossible, and the claim follows.
\end{proof}

\begin{claim}\label{claim2}\normalfont
The singularity $p_j$ is analytically linearizable.
\end{claim}
\begin{proof}[Proof of Claim 2]
By the reduction theorem and Claim~\ref{claim1}, $p_j$ is a reduced
nondegenerate singularity. Let $\lambda_j$ denote its eigenvalue
quotient, with the convention that the multiplier of the holonomy
$h_j$ of the separatrix $E_j$ is
\[
h_j'(0)=e^{2\pi\sqrt{-1}\lambda_j}.
\]
Since $h_j$ is conjugate to an element of the finite group $H(E_j)$,
it is periodic. Hence $h_j'(0)$ is a root of unity and
$\lambda_j\in\Q$. Since $p_j$ is reduced, its eigenvalue quotient
cannot be a positive rational number. Therefore
\[
\lambda_j=-\frac{\ell_j}{k_j},
\qquad
k_j,\ell_j\in\N,\qquad
\gcd(k_j,\ell_j)=1.
\]
By the classical holonomy linearization theorem for reduced
nondegenerate singularities (see \cite{MatteiMoussu1980}), a
singularity with negative rational eigenvalue quotient and periodic
separatrix holonomy is analytically linearizable. Consequently, there
are holomorphic coordinates $(x_j,y_j)$ centered at $p_j$ on a
neighborhood $B_j$ of $p_j$ in which
\[
\widetilde\F:
k_jx_j\,dy_j+\ell_jy_j\,dx_j=0.
\]
We may choose the coordinates so that
\[
E_j\cap B_j=\{y_j=0\}.
\]
If $p_j$ is a corner belonging to $E_i\cap E_j$, then the other
component is given by
\[
E_i\cap B_j=\{x_j=0\}.
\]
The function
\[
F_{p_j}(x_j,y_j)=x_j^{\ell_j}y_j^{k_j}
\]
is therefore a holomorphic first integral of $\widetilde\F$ in $B_j$.
\end{proof}

At each singular point $p\in E$, Claim~\ref{claim2} supplies a primitive
local first integral $Q_p=x^{\ell_p}y^{k_p}$. For every component $E_j$,
choose generators of
$\pi_1(E_j\setminus\Sing(\widetilde\F))$ and transport to a fixed
transversal both their holonomy maps and the finite invariance groups of
the restrictions of the $Q_p$ at the punctures. These maps generate a
finitely generated subgroup
\[
 \mathcal K_j<\Diff(\C,0).
\]
Every element of $\mathcal K_j$ sends a point to another point on the same
leaf.  Apart from the countable intersections with the finitely many leaves
that adhere to the origin, condition~\textup{(ii)} therefore gives finite
pseudo-orbits. Lemma~\ref{lemma:finite-holonomy} implies that
$\mathcal K_j$ is finite for every $j$. Lemma~\ref{lemma:tree-gluing}
therefore gives a single nonconstant holomorphic first integral $F$ on a
neighborhood $V$ of all of $E$.
\end{proof}

Let $F$ be the first integral in a neighborhood $V$ of $E$ obtained in
Lemma~\ref{Lemma:5}. The exceptional fiber of a resolution of a normal surface germ is
connected. Since every component of $E$ is invariant, $F$ is constant on the
regular part of each component; continuity at the crossing points shows that
$F$ has one constant value on all of $E$. Subtracting this value, we may
assume that $F|_E=0$. After shrinking $V$, properness of $\pi$ gives a
neighborhood $U_0$ of $0$ such that $\pi^{-1}(U_0)\subset V$. Because $\pi$
is biholomorphic away from $E$, the function
\[
 f:=F\circ\pi^{-1}
\]
is a holomorphic first integral on $U_0\setminus\{0\}$. The function $F$ is
bounded near the compact divisor $E$, so $f$ is locally bounded near $0$.
By the Riemann extension theorem for normal spaces,
$f$ extends uniquely to a holomorphic germ in $\OO_{X,0}$. It is
nonconstant because $F$ is nonconstant. This completes the proof of
Theorem~\ref{thm:A}.

\begin{exe}\label{exem2}
The normality hypothesis in Theorem~\ref{thm:A} is essential.  Consider
the reduced nonnormal surface
\[
 X=\{xy=0\}\subset(\C^3,0),
 \qquad X_{\sing}=\{x=y=0\},
\]
with irreducible components $X_1=\{x=0\}$ and $X_2=\{y=0\}$.  Define
\[
 V=x(2z+3z^2)\frac{\partial}{\partial x}
   +2zy\frac{\partial}{\partial y}
   -(x+y)\frac{\partial}{\partial z}.
\]
Since
\[
 V(xy)=xy(4z+3z^2),
\]
$V$ induces a holomorphic foliation $\F$ on $X$.

On $X_1$, with coordinates $(y,z)$, we have
\[
 V|_{X_1}=y\left(2z\frac{\partial}{\partial y}
                  -\frac{\partial}{\partial z}\right),
 \qquad F_1(y,z)=y+z^2,
\]
and $V(F_1)=0$.  On $X_2$, with coordinates $(x,z)$,
\[
 V|_{X_2}=x\left((2z+3z^2)\frac{\partial}{\partial x}
                  -\frac{\partial}{\partial z}\right),
 \qquad F_2(x,z)=x+z^2+z^3,
\]
and $V(F_2)=0$.  The factors $y$ and $x$ do not vanish on the respective
parts of $X_{\reg}$. Hence the leaves on each connected component of
$X_{\reg}$ are the connected components of the level sets of $F_1$ and
$F_2$, and are closed there.

A leaf can adhere to the origin only when it belongs to the zero level of
$F_1$ or $F_2$.  Thus precisely the two leaves
\[
 \{y+z^2=0\}\cap X_{\reg},
 \qquad
 \{x+z^2+z^3=0\}\cap X_{\reg}
\]
adhere to the origin.

We finally show that $\F$ has no nonconstant holomorphic first integral on
$X$.  A germ $f\in\mathcal O_{X,0}$ is a pair
\[
 (f_1(y,z),f_2(x,z))
\]
satisfying $f_1(0,z)=f_2(0,z)$. The pairs $(F_1,z)$ and $(F_2,z)$ are local
coordinate systems on the two smooth components. Since $f_1$ and $f_2$ are
constant along the respective level curves of $F_1$ and $F_2$, there are
one-variable holomorphic germs $\varphi$ and $\psi$ such that
\[
 f_1=\varphi(F_1),\qquad f_2=\psi(F_2).
\]
Compatibility on the singular axis becomes
\[
 \varphi(z^2)=\psi(z^2+z^3).
\]
The left-hand side is even, and therefore
\[
 \psi(z^2+z^3)=\psi(z^2-z^3).
\]
If the first nonconstant term of $\psi$ is $a_kt^k$, then, up to sign, the
first nonzero term of the difference is $2ka_kz^{2k+1}$, a contradiction.
Thus $\psi$, and then $\varphi$, is constant.  Consequently, both
dynamical conditions of Theorem~\ref{thm:A} hold, but no nonconstant
holomorphic first integral exists.  This proves genuine sharpness of the
normality hypothesis.
\end{exe}

\subsection{Existence of holomorphic first integrals in codimension one}

In this section we prove a natural extension of Theorem~\ref{thm:A} to
germs of codimension-one holomorphic foliations.
Recall from the introduction that
a germ $(X,0)\subset(\C^m,0)$ of a complex analytic variety of dimension
$n\geq2$ is \textit{normal-tree} if $(X,0)$ is normal and a nonempty
Zariski-open family of its
Bertini-type two-dimensional linear sections consists of normal-tree
analytic surfaces. Our framework is a normal-tree analytic germ $(X,0)$ of
dimension $n\geq2$, equipped with a germ of a codimension-one holomorphic
foliation $\F$.

\begin{teo2}\normalfont
Let $(X,0)$ be a germ of a normal-tree complex analytic variety of dimension $n\geq2$. Given a germ of a codimension-one holomorphic foliation $\F$ on $(X,0)$, assume, when $n\geq3$, that $\F$ is defined by a holomorphic integrable $1$-form $\omega$ with $\codim_X\Sing(\F)\geq2$ and that $((X,0),\F)$ is quotient-prolongation-admissible. Then the following conditions are equivalent:
\begin{enumerate}
\item[(i)] $\F$ admits a holomorphic first integral $f\in\OO_{X,0}$;
\item[(ii)] there is a neighborhood $U$ of $0$ in $X$ where the leaves of $\F|_U$ are closed subsets of $U\setminus\Sing(\F)$, and only finitely many of the leaves of $\F|_U$ adhere to $0\in U\subset X$.
\end{enumerate}
\end{teo2}

Observe that a leaf $L$ of $\F|_U$ which is closed in
$U\setminus\Sing(\F)$ is an analytic subvariety of
$U\setminus\Sing(\F)$ of dimension $n-1$. Since
\[
\dim\Sing(\F)\leq n-2<n-1=\dim L,
\]
the Remmert--Stein extension theorem implies that the closure
$\overline L^{\,U}$ is an analytic subvariety of $U$. Notice that, when
$L$ accumulates on $\Sing(\F)$, it is its closure, rather than $L$
itself, that is an analytic subvariety of $U$.

\subsection{Generic surface sections for the higher-dimensional arguments}

After choosing a small representative of the germ, one may consider $X$ as a closed analytic subset of an open neighborhood $\Omega$ of the origin $0\in\C^m$:
\[
X\simeq\{z\in\Omega;\;f_1(z)=\cdots=f_r(z)=0\}
\]
for some $f_1,\ldots,f_r\in\OO(\Omega)$. Let $n=\dim X\geq2$.

If $n=2$, then we are in the situation considered in
Theorem~\ref{thm:A}. Assume therefore that $n\geq3$.

Since $(X,0)$ is a normal-tree germ, by definition we may choose a
Bertini-type generic linear subspace $H\subset\C^m$ of codimension
$n-2$, passing through $0$, such that
\[
(S,0):=(X\cap H,0)
\]
is a normal-tree complex analytic surface germ. Since the pair is
quotient-prolongation-admissible, we choose $H$ in the intersection of the
corresponding nonempty Zariski-open families. Thus $H$ is in general position
with respect to $((X,0),\F)$ and satisfies the smooth
quasi-\'etale-cover conditions of Theorem~\ref{thm:E}. The restriction
\[
\F_S:=\F|_S
\]
is a holomorphic foliation by curves on $(S,0)$ with an isolated
singularity at $0$.

For context, we recall the following Bertini-type normality result.

\begin{teo}[Bertini-type normality theorem]\normalfont
Let $(X,0)\subset(\C^m,0)$ be a normal Cohen--Macaulay complex
analytic germ of dimension $n\geq3$. Then a generic complex hyperplane
$H\subset\C^m$ passing through $0$ has the property that
$(X\cap H,0)$ is a normal complex analytic germ of dimension $n-1$.
More generally, for a single hyperplane section, the
Cohen--Macaulay assumption may be replaced by the local depth
condition
\[
\operatorname{depth}\mathcal O_{X,0}\geq3.
\]
\end{teo}

This is a local Bertini theorem with a fixed base point; see
\cite{Flenner1977,Manaresi1982}. The depth condition is needed to guarantee
Serre's condition $S_2$ at $0$. In the present proof, however, the
normality and the normal-tree property of $(S,0)$ follow directly
from the definition of a normal-tree germ, rather than from normality
of $(X,0)$ alone.

\subsection{Formal and holomorphic prolongation}

\begin{proof}[Proof of Theorem~\ref{theorem:formal-prolongation}]
Put $I=\mathcal I_{Y,0}$ and $M_k=I^k/I^{k+1}$, and set
$U=Y\setminus\{0\}$. At every point $p\in U$, both $X$ and the foliation
are nonsingular, and $Y$ is transverse to the foliation. The classical
Mattei--Moussu prolongation theorem therefore gives a unique germ $F_p$ of a
holomorphic first integral near $p$ satisfying $F_p|_Y=f_0$. On overlaps,
uniqueness implies that the jets of the local prolongations agree along $Y$.
Consequently, for every $k\geq1$, their order-$k$ transverse terms glue to a
section
\[
 s_k\in H^0(U,M_k|_U).
\]
If $\operatorname{depth}_{\mathfrak m_Y}M_k\geq2$, then
$H^0_{\{0\}}(M_k)=H^1_{\{0\}}(M_k)=0$. The local-cohomology exact sequence
thus identifies
\[
 M_k\simeq H^0(U,M_k|_U).
\]
We spell out the induction. Suppose that a class
$f^{(k)}\in\OO_{X,0}/I^{k+1}$ has been constructed, restricts to $f_0$, and
satisfies the first-integral equation modulo order $k+1$ along $Y$.  Here
the differential filtration is
\[
 \mathcal F_I^r\Omega^p_{X,0}
 :=I^r\Omega^p_{X,0}
   +d(I^{r+1})\wedge\Omega^{p-1}_{X,0}.
\]
It satisfies $d(I^{r+1})\subset\mathcal F_I^r\Omega^1_{X,0}$, so exterior
differentiation is continuous for these filtrations. Choose
any lift $\widetilde f^{(k+1)}$ to
$\OO_{X,0}/I^{k+2}$. On $U$, compare this lift with the order-$(k+1)$ jet of
the local Mattei--Moussu prolongation. Their difference is a uniquely
determined section
\[
 s_{k+1}\in H^0(U,M_{k+1}|_U).
\]
The depth hypothesis extends $s_{k+1}$ uniquely to an element of
$M_{k+1}$. Adding this element to $\widetilde f^{(k+1)}$ produces a class
$f^{(k+1)}$ satisfying the first-integral equation modulo $I^{k+2}$ and
reducing to $f^{(k)}$. More explicitly, on $U$ the corrected class has the
same order-$(k+1)$ jet as the local holomorphic prolongations. Hence the
class of $\omega\wedge df^{(k+1)}$ in
$\Omega^2_{X,0}/\mathcal F_I^{k+1}\Omega^2_{X,0}$ restricts to zero on
$U$. By the injectivity hypothesis in
Theorem~\ref{theorem:formal-prolongation}, this class is zero. Therefore
\[
 \omega\wedge df^{(k+1)}=0
 \quad\bmod\mathcal F_I^{k+1}\Omega^2_{X,0}.
\]
This also shows
that the equation is independent of the arbitrary lift and is compatible
with the class constructed at the preceding stage. Starting with
$f^{(0)}=f_0$, induction gives a
compatible system of solutions modulo $I^{k+1}$ for all $k$. Its inverse
limit is an $I$-adic formal first integral $\widehat f_I$ restricting to
$f_0$. The differential equation holds in
$\varprojlim_r\Omega^2_{X,0}/\mathcal F_I^r\Omega^2_{X,0}$ because it holds
at every finite order. Since
$I\subset\mathfrak m_{X,0}$, the maps
$\OO_{X,0}/I^k\to\OO_{X,0}/\mathfrak m_{X,0}^k$ induce a homomorphism from the
$I$-adic completion to $\widehat\OO_{X,0}$; the image of $\widehat f_I$ is a
formal first integral. It is nonconstant. Indeed, if its difference from
$f_0(0)$ vanished in the $\mathfrak m_{X,0}$-adic completion, restriction
to $Y$ would give
\[
 f_0-f_0(0)\in\bigcap_{r\geq1}\mathfrak m_{Y,0}^r=0
\]
by the Krull intersection theorem, contrary to the choice of $f_0$.
\end{proof}

\begin{proof}[Proof of Theorem~\ref{thm:E}]
Let
$\pi:(\widetilde X,\widetilde0)\to(X,0)$ be the finite quasi-\'etale quotient
map and put $\widetilde Y=(\pi^{-1}(Y))_{\mathrm{red}}$. The pullback
\[
 \widetilde f_0=f_0\circ\pi|_{\widetilde Y}
\]
is a holomorphic first integral of the restricted saturated foliation
$\widetilde\F|_{\widetilde Y}$. Since $\widetilde X$ and $\widetilde Y$ are
smooth, the classical Mattei--Moussu prolongation theorem gives a unique
holomorphic first integral $\widetilde f\in\OO_{\widetilde X,\widetilde0}$
such that $\widetilde f|_{\widetilde Y}=\widetilde f_0$.

For $g\in G$, the germ $g^*\widetilde f$ is another first integral of the
$G$-invariant pullback foliation. Moreover, because $\widetilde Y$ is
$G$-invariant and $\pi\circ g=\pi$, it has the same restriction
$\widetilde f_0$ to $\widetilde Y$. Uniqueness gives
$g^*\widetilde f=\widetilde f$. Hence $\widetilde f$ descends through the
finite quotient to a germ $f\in\OO_{X,0}$. The first-integral identity holds
on the dense locus where $\pi$ is \'etale and therefore holds everywhere in
the saturated interpretation. Finally,
\[
 (\pi|_{\widetilde Y})^*(i^*f)
 =\widetilde f|_{\widetilde Y}
 =(\pi|_{\widetilde Y})^*f_0.
\]
Since the finite map $\pi|_{\widetilde Y}$ is surjective, its pullback on
holomorphic functions is injective, and $i^*f=f_0$.
\end{proof}

\subsection{Proof of Theorem \ref{thm:B}}

\begin{proof}[Proof of Theorem B]
Retain the generic surface section $(S,0)$ fixed above.

\noindent\textup{(i)$\Rightarrow$(ii).}
Let $f\in\OO_{X,0}$ be a nonconstant holomorphic first integral for
$\F$, and choose a sufficiently small representative
$f\colon U\to\C$. Every leaf of $\F$ in
$U\setminus\Sing(\F)$ is a connected component of a fiber of $f$ and
is therefore closed in $U\setminus\Sing(\F)$. A leaf adhering to $0$
is contained in the fiber $f^{-1}(f(0))$. This analytic hypersurface
germ has only finitely many irreducible components, and the regular
part of each component has only finitely many connected components in
a sufficiently small representative. Hence only finitely many leaves
adhere to $0$.

\medskip
\noindent\textup{(ii)$\Rightarrow$(i).}
Assume condition \textup{(ii)} and choose the generic normal-tree surface
section $(S,0)$ fixed above. A leaf of the restricted foliation
$\F_S$ is a connected component of $S\cap L$, where $L$ is a leaf of
$\F$. Since the ambient leaves are closed in
$U\setminus\Sing(\F)$ and the section is in general position, the
leaves of $\F_S$ are closed in $S\setminus\{0\}$.

If a leaf of $\F_S$ adheres to $0$, then its ambient leaf $L$ also
adheres to $0$. There are only finitely many such ambient leaves. For
each of them, the Remmert--Stein theorem implies that
$\overline L^{\,U}$ is an analytic hypersurface germ. The intersection
\[
\overline L^{\,U}\cap S
\]
is an analytic curve germ and has only finitely many irreducible
branches. Away from the origin, the regular part of each branch is
contained in a leaf of $\F_S$; after choosing a sufficiently small
representative, it has only finitely many connected components.
Consequently, only finitely many leaves of $\F_S$ adhere to $0$. Thus
$\F_S$ satisfies condition \textup{(ii)} of Theorem~A.
By Theorem~A, $\F_S$ admits a holomorphic first integral. The
Prolongation Theorem, Theorem~\ref{thm:E}, then yields a holomorphic first integral for
$\F$ on $(X,0)$.
\end{proof}
\section{Formal first integrals}\label{sec:formal}

\subsection{Formal functions on analytic spaces}

Let $(X,0)$ be a germ of complex analytic variety.
\begin{itemize}
\item $\OO_{X,0}$ denotes the local ring of holomorphic function germs on $X$ at $0\in X$;
\item $\mm_0\subset\OO_{X,0}$ denotes the maximal ideal of $\OO_{X,0}$, i.e., $\mm_0=\{f\in\OO_{X,0};\ f(0)=0\}$;

\item $\mm_0^k$ is the ideal of $\OO_{X,0}$ generated by products
$f_1\cdots f_k$ of germs $f_1,\ldots,f_k\in\mm_0$;
\item we have $\cdots\subset \mm_0^{k+1}\subset \mm_0^k\subset\cdots\subset\mm_0^2\subset\mm_0$.
\end{itemize}

The ring of formal functions on $(X,0)$ is the $\mm_0$-adic completion
\[
  \widehat{\OO}_{X,0}:=\varprojlim_k \OO_{X,0}/\mm_0^{k+1}.
\]
Thus, an element $\wh f\in\widehat{\OO}_{X,0}$ is represented by a sequence
$(f_0,f_1,f_2,\ldots)$, where $f_k\in \OO_{X,0}/\mm_0^{k+1}$, under the compatibility condition $f_{k+1}\equiv f_k \pmod{\mm_0^{k+1}}$.

\begin{exe}\label{exem4}
Consider the embedded case $(X,0)\subset (\CC^N,0)$, i.e.,
  \[(X,0)=\{f_1=\cdots=f_r=0\}\subset(\CC^N,0)
\]
for some $f_1,\ldots,f_r\in\OO_{\CC^N,0}$.

We know that $\OO_{X,0}$ consists of the restrictions to
$\{f_1=\cdots=f_r=0\}$ of germs $F\in\OO_{\CC^N,0}$, so that, if we consider coordinates $(z_1,\ldots,z_N)\in(\CC^N,0)$, then
\[
  \OO_{X,0}\simeq \frac{\CC\{z_1,\ldots,z_N\}}{\langle f_1,\ldots,f_r\rangle}.
\]
This shows that the completion $\widehat{\OO}_{X,0}$ is given by
\[
  \widehat{\OO}_{X,0}
  \simeq
  \frac{\CC[[z_1,\ldots,z_N]]}{\langle f_1,\ldots,f_r\rangle^\wedge},
\]
where
\[
  \CC[[z_1,\ldots,z_N]]=
  \left\{\widehat f=\sum_{(i_1,\ldots,i_N)\in\NN^N}
  a_{i_1\ldots i_N} z_1^{i_1}\cdots z_N^{i_N},\quad
  a_{i_1\ldots i_N}\in\CC\right\}
\]
and $\langle f_1,\ldots,f_r\rangle^\wedge$ is the ideal of $\CC[[z_1,\ldots,z_N]]$ generated by $f_1,\ldots,f_r$. In particular, an element $\wh f\in\widehat{\OO}_{X,0}$ is represented by a formal power series
\[
  \widehat f(z)=\sum_{I\in\NN^N} a_I z^I\in\CC[[z_1,\ldots,z_N]]
\]
once considered modulo the ideal that defines $X$. The inclusion $\OO_{X,0}\to \widehat{\OO}_{X,0}$ is then obtained by Taylor expansion.
\end{exe}

\subsection{Lifting formal functions}

Let $(X,0)$ be a normal two-dimensional complex analytic variety germ and let $\pi\colon\wt X\to X$ be a resolution of $(X,0)$, with exceptional divisor $E=\pi^{-1}(0)\subset\wt X$. Fix a point $\wt p\in E$. We consider the natural pull-back homomorphism
\[\begin{array}{cccc}
  \pi_{\wt p}^{\#}\colon&\OO_{X,0}&\to& \OO_{\wt X,\wt p}\\
   &f&\mapsto & (f\circ\pi)_{\wt p}.\end{array}
\]
Because $\pi(\wt p)=0$, we have $\pi_{\wt p}^{\#}(\mm_0)\subset \mm_{\wt p}$, for the ideals $\mm_0\subset\OO_{X,0}$ and $\mm_{\wt p}\subset\OO_{\wt X,\wt p}$. Thus $\pi_{\wt p}^{\#}$ is continuous in the $\mm$-adic topologies in $\OO_{X,0}$ and $\OO_{\wt X,\wt p}$, and admits a unique extension to the completions
\[
  \widehat{\pi_{\wt p}^{\#}}\colon
  \widehat{\OO}_{X,0}\to \widehat{\OO}_{\wt X,\wt p}.
\]
Given $\wh f\in\widehat{\OO}_{X,0}$, we obtain the \textit{pull-back} or lifting of $\wh f$ via $\pi\colon\wt X\to X$ defined by
\[
  (\pi^*\wh f)_{\wt p}:=\widehat{\pi_{\wt p}^{\#}}(\wh f),
  \qquad \forall\wt p\in E.
\]
A global way of introducing $\pi^*\wh f$ is as follows. Take a small
neighborhood $\wt U\subset\wt X$ of $E$ and put
$\mathcal J_{\wt U}:=\mm_0\OO_{\wt U}$, whose support is $E$. For
$k\in\NN$, the finite jet
$f_k\in \OO_{X,0}/\mm_0^{k+1}$ pulls back to
$\pi^*f_k\in \OO_{\wt U}/\mathcal J_{\wt U}^{k+1}$.
The compatibility of the jets $f_k$ implies the compatibility of the $\pi^*f_k$, and therefore
\[
  \pi^*\wh f=(\pi^*f_0,\pi^*f_1,\ldots)
  \in \varprojlim_k \OO_{\wt U}/\mathcal J_{\wt U}^{k+1}.
\]
Thus, ``$\pi^*\wh f$ is a formal function along the exceptional divisor.'' We shall write $\pi^*\wh f\in\widehat{\OO}_{\wt X,E}$.

Given a holomorphic derivation $v\colon\OO_{X,0}\to\OO_{X,0}$, by the Leibniz rule
\[
  v(fg)=v(f)g+f v(g),
\]
we have $v(\mm^{k+1})\subset \mm^k$,
where $\mm=\mm_0\subset\OO_{X,0}$ is the maximal ideal. This means that $v$ is continuous for the $\mm$-adic topology. Take an element $\wh f\in\widehat{\OO}_{X,0}$ as represented by
$\wh f=(f_0,f_1,\ldots)$,
   $f_k\in\OO_{X,0}/\mm^{k+1}$. For each jet $f_k$, we choose a representative $F_k\in\OO_{X,0}$ and define the $k$-th jet of $\widehat v(\wh f)$ as
\[
  \widehat v(\wh f)_k:=v(F_{k+1})\pmod{\mm^{k+1}}.
\]
Then we define
\[
  \widehat v(\wh f)\in \varprojlim_k \OO_{X,0}/\mm^{k+1}
  =\widehat{\OO}_{X,0}
\]
by these jets.

If $\FF$ is a holomorphic foliation germ, then a \textit{formal first
integral} for $\FF$ is a nonconstant formal function
$\wh f\in\widehat{\OO}_{X,0}$ such that
\[
\widehat v(\wh f)=0
\]
for every holomorphic derivation $v$ tangent to $\FF$.
Here nonconstant means that $\wh f\notin\CC$, or, equivalently,
\[
\wh f-\wh f(0)\neq0.
\]
When the tangent sheaf of $\FF$ is locally generated by one derivation
$v$, it is enough to require $\widehat v(\widehat f)=0$.

For a codimension-one foliation on a higher-dimensional germ defined by a
holomorphic integrable $1$-form $\omega$, we shall use the following
differential-form formulation: a formal first integral
is a nonconstant $\widehat f\in\widehat{\mathcal O}_{X,0}$ satisfying
\[
 \omega\wedge d\widehat f=0
\]
in the completed module of K\"ahler $2$-forms. This is the notion used in
Theorem~\ref{thm:D}.

Let $\pi_1\colon X^*\to X$ be a resolution of $(X,0)$ with exceptional divisor
$E^*=\pi_1^{-1}(0)\subset X^*$.

If $X$ is normal, then we may consider a holomorphic vector field $v$
defined in a neighborhood $U$ of $0$ in $X$ and defining $\FF$ in
$U$. The lifted foliation $\FF^*$ is regular on $X^*\setminus E^*$,
and its singular set is a finite subset of $E^*$. Applying Seidenberg's
resolution theorem to $\FF^*$, we obtain a proper modification
\[
\pi_2\colon\wt X\longrightarrow X^*.
\]
Setting
\[
\pi:=\pi_1\circ\pi_2\colon\wt X\longrightarrow X,
\]
we obtain a resolution of the pair $((X,0),\FF)$, with exceptional
divisor $E:=\pi^{-1}(0)\subset\wt X$ and lifted foliation
$\wt\FF=\pi^*(\FF)$, whose singular set is finite and contained in
$E$.

\begin{lema}[Formal injectivity for a proper modification]
\label{lemma:formal-injectivity}\normalfont
Let $\rho:(Z,E)\to(X,0)$ be a proper modification, where $(X,0)$ is a
normal complex analytic germ and $Z$ is nonsingular. For every $p\in E$,
the homomorphism
\[
 \widehat{\rho_p^\#}:\widehat{\mathcal O}_{X,0}
 \longrightarrow\widehat{\mathcal O}_{Z,p}
\]
is injective.
\end{lema}

\begin{proof}
Put $A=\mathcal O_{X,0}$, $B=\mathcal O_{Z,p}$, and
$\varphi=\rho_p^\#:A\to B$. These are local complex analytic algebras;
$A$ is normal and $B$ is regular. Since $\rho$ is a modification, every
neighborhood of $p$ meets the open set on which $\rho$ is biholomorphic.
Consequently, the generic rank of the morphism of complex analytic algebras
$\varphi$ is maximal:
\[
 r(\varphi)=\dim A.
\]
The generic, formal, and analytic ranks satisfy
\[
 r(\varphi)\leq r_F(\varphi)\leq r_A(\varphi)\leq\dim A.
\]
Hence $r_F(\varphi)=\dim A$. By the definition of formal rank
\cite[Definition~1.3]{BelottoCurmiRond2021},
\[
 r_F(\varphi)=\dim\widehat A-
 \operatorname{ht}\ker\widehat\varphi.
\]
The analytic local ring $A$ is excellent; since it is normal, its completion
$\widehat A$ is a normal local ring and hence an integral domain. Therefore
the equality $r_F(\varphi)=\dim\widehat A$ implies that
$\ker\widehat\varphi$ is a height-zero ideal in a domain, and consequently
$\ker\widehat\varphi=0$. Thus
$\widehat{\rho_p^\#}:\widehat A\to\widehat B$ is injective.

For broader background on formal functions and formal embeddings under
proper morphisms, see also \cite{HironakaMatsumura1968}.
\end{proof}

\begin{propo}\label{prop4}\normalfont
A formal first integral $\widehat f\in\widehat{\OO}_{X,0}$ lifts by $\pi\colon\wt X\to X$ to a formal function $\pi^*\widehat f\in\widehat{\OO}_{\wt X, E}$ which is a formal first integral for the lifted foliation $\wt\FF=\pi^*\FF$, along the exceptional divisor $E=\pi^{-1}(0)$.
\end{propo}
\begin{proof}
Fix a point $\widetilde p\in E$ and consider the local homomorphism
\[
\begin{array}{cccl}
\pi_{\widetilde p}^{\#}\colon
&\OO_{X,0}&\longrightarrow&\OO_{\widetilde X,\widetilde p}\\
&f&\longmapsto&(f\circ\pi)_{\widetilde p}.
\end{array}
\]
Since $\pi(\widetilde p)=0$, we have
\[
\pi_{\widetilde p}^{\#}(\mathfrak m_0)
\subset\mathfrak m_{\widetilde p}.
\]
Thus, $\pi_{\widetilde p}^{\#}$ extends to the completions:
\[
\widehat{\pi_{\widetilde p}^{\#}}\colon
\widehat{\OO}_{X,0}
\longrightarrow
\widehat{\OO}_{\widetilde X,\widetilde p}.
\]
The formal pull-back is defined by
\[
(\pi^*\widehat f)_{\widetilde p}
:=
\widehat{\pi_{\widetilde p}^{\#}}(\widehat f).
\]
Applying Lemma~\ref{lemma:formal-injectivity} to the present map shows that
$\widehat{\pi_{\widetilde p}^{\#}}$ is injective.  The pointwise pullbacks
are compatible on overlaps.  They therefore define a section of the
completion along the scheme-theoretic fibre defined by
$\mathfrak m_0\mathcal O_{\widetilde X}$.  Since this ideal and the reduced
ideal of $E$ have the same radical, their adic topologies are equivalent;
hence this is also a formal function along $E$. After subtracting the
constant term, put
\[
\widehat f_0:=\widehat f-\widehat f(0)\neq0.
\]
Then
\[
(\pi^*\widehat f_0)_{\widetilde p}
=\widehat{\pi_{\widetilde p}^{\#}}(\widehat f_0)\neq0.
\]
Consequently, $\pi^*\widehat f$ is nonconstant.

Let $\widetilde v$ be a local holomorphic generator of the saturated
strict transform $\widetilde\F$ near $\widetilde p$. On
$\widetilde X\setminus E$, the vector fields $\widetilde v$ and the
pull-back of $v$ define the same line field. Consequently, there
exists a nonzero meromorphic germ $\rho$ at $\widetilde p$ such that
\[
\widetilde v(\pi^*\varphi)
=
\rho\,\pi^*(v(\varphi)),
\qquad
\forall\varphi\in\OO_{X,0}.
\]
After clearing denominators, there exist nonzero holomorphic germs
$a,b\in\OO_{\widetilde X,\widetilde p}$ such that
\[
a\,\widetilde v(\pi^*\varphi)
=
b\,\pi^*(v(\varphi)),
\qquad
\forall\varphi\in\OO_{X,0}.
\]
The same identity holds after passing to the completions. Since
$\widehat v(\widehat f)=0$, we obtain
\[
a\,\widehat{\widetilde v}
\bigl(\pi^*\widehat f\bigr)=0.
\]
The completed local ring
$\widehat{\OO}_{\widetilde X,\widetilde p}$ is an integral domain,
and $a\neq0$. Therefore
\[
\widehat{\widetilde v}
\bigl(\pi^*\widehat f\bigr)=0.
\]
Since $\widetilde p\in E$ was arbitrary, $\pi^*\widehat f$ is a
formal first integral of $\widetilde\F$ along $E$.
\end{proof}

\subsection{Auxiliary Lemmas}

\begin{lema}
\label{lemma:6}\normalfont
Let $G<\Diff(\C,0)$ be a group of germs of one-variable complex analytic
\textup{(not necessarily finitely generated)} diffeomorphisms at $0\in\C$.
If there is a non-constant formal function
$\hat f\in \widehat{\OX}_{\C,0}$, $\hat f(0)=0$, such that
\[
        \hat f\circ \varphi = \hat f, \qquad \forall\varphi\in G,
\]
then $G$ is finite.
\end{lema}

\begin{proof}
Let
\[
m:=\ord_0(\widehat f).
\]
We may write
\[
\widehat f(z)=z^m\widehat u(z),
\]
where $\widehat u\in\C[[z]]$ is a unit. Since every unit in
$\C[[z]]$ admits an $m$-th root, choose a unit
$\widehat v\in\C[[z]]$ such that
\[
\widehat u(z)=\widehat v(z)^m.
\]
Define
\[
\widehat\psi(z):=z\widehat v(z).
\]
Since $\widehat\psi'(0)=\widehat v(0)\neq0$, the series
$\widehat\psi$ is a formal change of coordinate. Moreover,
\[
\widehat f(z)=\widehat\psi(z)^m.
\]
Thus, in the new formal coordinate $w=\widehat\psi(z)$,
\[
\widehat f\circ\widehat\psi^{-1}(w)=w^m.
\]

For $\varphi\in G$, set
\[
\widetilde\varphi
:=
\widehat\psi\circ\varphi\circ\widehat\psi^{-1}.
\]
The identity $\widehat f\circ\varphi=\widehat f$ gives
\[
\widetilde\varphi(w)^m=w^m.
\]
Therefore
\[
\widetilde\varphi(w)^m-w^m
=\prod_{\xi\in\mu_m}
\bigl(\widetilde\varphi(w)-\xi w\bigr)=0.
\]
Since $\C[[w]]$ is an integral domain, one of these factors must
vanish. Thus, for some $\xi_\varphi\in\mu_m$,
\[
\widetilde\varphi(w)=\xi_\varphi w
\]
and $\xi_\varphi^m=1$.

Consequently, the map
\[
G\longrightarrow\mu_m,
\qquad
\varphi\longmapsto\xi_\varphi,
\]
is an injective group homomorphism. Therefore
\[
|G|\leq m,
\]
and $G$ is finite.
\end{proof}

The following lemma will be useful: Let $\pi\colon(\widetilde X,E)\to (X,0)$
be a resolution of the pair $(X,\F)$, with lifted foliation
$\widetilde\F=\pi^*\F$. Assume that we have a formal function $\widehat F$ along $E$, which is a
formal first integral for $\widetilde\F$. Choose a component $E_j\subset E$,
a regular point
\[
       q_j\in L_j:=E_j\setminus \Sing(\widetilde\F),
\]
and a transverse section $\Sigma_j$, $\Sigma_j\cap E=\{q_j\}$.

\begin{lema}
\label{lemma:7}\normalfont A holonomy map $h\in \Hol(\widetilde\F,L_j,\Sigma_j,q_j)$ leaves invariant the restriction $\widehat F|_{\Sigma_j}$.
\end{lema}
\begin{proof}
Cover a path defining $h$ by finitely many flow boxes contained in the
regular locus of $\widetilde\F$ and meeting $E$. In each flow box, let
$\tilde v$ be a holomorphic vector field generating $\widetilde\F$ and
let $\varphi_t$ denote its local flow. Since $E$ is
$\widetilde\F$-invariant, $\varphi_t$ preserves the ideal of $E$ and
therefore induces an automorphism $\varphi_t^*$ of the formal completion
along $E$. Since $\widehat F$ is a formal first integral,
$\tilde v(\widehat F)=0$ in this completion. Hence, formally,
\[
       \frac{d}{dt}\,\varphi_t^*\widehat F
       =\varphi_t^*\bigl(\tilde v(\widehat F)\bigr)=0,
\]
and therefore $\varphi_t^*\widehat F=\widehat F$.

The holonomy map $h$ is obtained by following the leaves of $\widetilde\F$
and composing finitely many such local flow maps between transversals;
therefore
\[
       \widehat F|_{\Sigma_j}\circ h=\widehat F|_{\Sigma_j}.       \qedhere
\]
\end{proof}

Consequently, we may conclude:

\begin{propo}\label{Proposition:6}\normalfont
Let $\F$ be a germ of holomorphic foliation on a germ of complex analytic
normal space $(X,0)$ of dimension two. Assume that $\F$ admits a formal
first integral $\hat f\in\widehat\OX_{X,0}$. Then:
\begin{enumerate}[label=\textup{(\arabic*)}]
\item $\F$ is non-dicritical;
\item $\widetilde\F$ exhibits no saddle-nodes;
\item every singularity $\tilde p\in\Sing(\widetilde\F)\subset E$ is
analytically linearizable of the form
\[
       kx\,dy+l y\,dx=0,\qquad k,l\in\N,\quad \gcd(k,l)=1,
\]
where
\[
       \{y=0\}\subset E\subset \{xy=0\};
\]
\item writing $E=\displaystyle\bigcup_{j=1}^r E_j$ in compact components,
the holonomy group of every component is finite:
\[
H(E_j)\subset\Diff(\Sigma_j,q_j),
\]
where
\[
       q_j\in E_j\setminus\Sing(\widetilde\F), \qquad
       \Sigma_j\pitchfork E, \qquad
       \Sigma_j\cap E=\Sigma_j\cap E_j=\{q_j\},
\]
for every choice of $q_j,\Sigma_j$ and every $j\in\{1,\ldots,r\}$.
\end{enumerate}
\end{propo}

\begin{proof}
\textup{(1)} After subtracting its constant term, we may assume that
\[
\widehat f(0)=0.
\]
Put
\[
\widehat F:=\pi^*\widehat f.
\]
Then $\widehat F$ vanishes along every component of
$E=\pi^{-1}(0)$.

Fix a component $E_j$ and choose a generic smooth point
$q\in E_j$. Let $x=0$ be a reduced local equation for $E_j$ near
$q$. Since $\widehat F$ is nonzero, we may write
\[
\widehat F=x^m\widehat U,
\qquad m\geq1,
\]
where $\widehat U$ is not divisible by $x$. Let $\widetilde v$ be a
local generator of $\widetilde\F$. Since $\widehat F$ is a formal
first integral,
\[
0=\widehat{\widetilde v}(\widehat F)
=x^{m-1}\left(
m\widehat U\,\widetilde v(x)
+x\widehat{\widetilde v}(\widehat U)
\right).
\]
Reducing the resulting identity modulo $x$, we obtain
\[
m\,(\widehat U|_{E_j})\,
(\widetilde v(x)|_{E_j})=0.
\]
At a generic point of $E_j$, $\widehat U|_{E_j}\neq0$. Hence
\[
\widetilde v(x)|_{E_j}=0,
\]
on a nonempty open subset of $E_j$. Since
$\widetilde v(x)|_{E_j}$ is holomorphic on the regular part of the
irreducible component $E_j$, the identity theorem gives
\[
\widetilde v(x)|_{E_j}\equiv0.
\]
Equivalently, $\widetilde v(x)\in(x)$ along $E_j$, and therefore
$\widetilde v$ is tangent to the entire component $E_j$. Thus every
component of $E$ is $\widetilde\F$-invariant, and consequently
$\F$ is non-dicritical.

\textup{(2)} Given a singularity
$\tilde p\in\Sing(\widetilde\F)\subset E$, the formal function
$(\pi^*\hat f)_{\tilde p}\in\widehat\OX_{\widetilde X,\tilde p}$ is a
formal first integral for the germ $\widetilde\F_{\tilde p}$ of the
foliation $\widetilde\F$ at $\tilde p\in\widetilde X$.
By Mattei--Moussu, Theorem A, page 472, the germ $\widetilde\F_{\tilde p}$
admits a nonconstant holomorphic first integral. Since $\pi$ is a
resolution of the pair $((X,0),\F)$, every singularity of
$\widetilde\F$ is simple, and hence is either a reduced nondegenerate
singularity or a saddle-node. A saddle-node cannot admit a nonconstant
holomorphic first integral. Therefore $\widetilde\F$ has no saddle-nodes,
which proves \textup{(2)}.

It remains to identify the nondegenerate singularities. A reduced
nondegenerate singularity admitting a holomorphic first integral has two
analytic separatrices and a negative rational quotient of eigenvalues.
Thus, after choosing relatively prime integers $k,l\in\N$ and suitable
analytic coordinates, a primitive holomorphic first integral can be
written as
\[
F(x,y)=x^l y^k.
\]
Consequently, the foliation is analytically equivalent to
\[
kx\,dy+l y\,dx=0.
\]
Since the exceptional divisor is invariant and has normal crossings, the
coordinates may be chosen so that one of its local branches is
$\{y=0\}$ and
\[
\{y=0\}\subset E\subset\{xy=0\}.
\]
This proves \textup{(3)}.

We now prove \textup{(4)}. Following the notation introduced above in the
proof, we consider a component $E_j$ of $E$ and choose a generic
nonsingular point $q_j\in E_j\setminus\Sing(\widetilde\F)$,
and a transverse disk $\Sigma_j$ centered at
$\{q_j\}=\Sigma_j\cap E_j$.
Put $\widehat F:=\pi^*\widehat f$ and denote by
$\widehat F|_{\Sigma_j}$ its restriction to the formal transverse
section $(\Sigma_j,q_j)$.
For every $h\in H(E_j)$, Lemma~\ref{lemma:7} gives
\[
       \bigl(\widehat F|_{\Sigma_j}\bigr)\circ h
       =\widehat F|_{\Sigma_j}.
\]
Choose a local coordinate $x$ transverse to $E_j$, so that
$E_j=\{x=0\}$, and write
\[
\widehat F=x^{m_j}\widehat U,
\qquad m_j\geq1,
\]
where $\widehat U$ is not divisible by $x$. We choose $q_j$ outside
the zero set of the leading coefficient $\widehat U|_{E_j}$. Hence
$\widehat U|_{\Sigma_j}$ is a unit, and
\[
\operatorname{ord}_{q_j}
\bigl(\widehat F|_{\Sigma_j}\bigr)=m_j\geq1.
\]
Thus $\widehat F|_{\Sigma_j}$ is nonzero, nonconstant, and vanishes at
$q_j$. Lemma~\ref{lemma:6} implies that
$H(E_j)$ is a finite group. For any
other regular point $q_j'\in E_j\setminus\Sing(\widetilde\F)$, transport
along the connected leaf
$E_j\setminus\Sing(\widetilde\F)$ conjugates the corresponding
holonomy group to the one just considered. Therefore the
holonomy group is finite for every choice of a regular
base point and transverse section on $E_j$.
\end{proof}

\subsection{Foliations admitting formal first integrals: Theorem \ref{thm:C}}

\begin{teo3}\normalfont
Let $\F$ be a germ of holomorphic foliation by curves on a germ of a
normal-tree complex analytic surface $(X,0)$. If $\F$ admits a formal
first integral $\hat f\in\widehat\OX_{X,0}$, then $\F$ admits a holomorphic
first integral $f\in\OX_{X,0}$.
\end{teo3}
\begin{proof}
Consider a resolution of singularities $\pi\colon\widetilde X\to X$
for the pair $(X,\F)$, with exceptional divisor $E=\pi^{-1}(0)$.
Given a formal first integral $\hat f\in\widehat\OX_{X,0}$ for the foliation $\F$, there is a lift
$\pi^*\hat f\in\widehat\OX_{\widetilde X,E}$ of $\hat f$ to a formal first
integral for the foliation $\widetilde\F=\pi^*(\F)$ along $E$
(cf. Proposition~\ref{prop4}). Moreover, according to Proposition~\ref{Proposition:6} we have:
\begin{enumerate}[label=\textup{(\arabic*)}]
\item $\F$ is non-dicritical;
\item every singularity $\tilde p\in\Sing(\widetilde\F)\subset E$ is
analytically linearizable of the form
\[
       kx\,dy+l y\,dx=0,
       \qquad k,l\in\N,\quad \gcd(k,l)=1.\]
where $\{y=0\}\subset E\subset\{xy=0\}$;
\item the holonomy group of every irreducible component of $E$ is finite.
\end{enumerate}
Because $(X,0)$ is normal-tree, the dual graph of $E$ is a finite tree.
At every singular point $p\in E$,
\textup{(2)} gives a primitive holomorphic first integral
$Q_p=x^{l_p}y^{k_p}$.  The formal factorization theorem of
Mattei--Moussu applied at $p$ gives a formal one-variable series
$\widehat\psi_p$ such that
\[
 (\pi^*\widehat f)_p=\widehat\psi_p\circ Q_p.
\]
The component holonomies preserve $\pi^*\widehat f$ by
Lemma~\ref{lemma:7}, while the puncture invariance groups preserve it by
the displayed factorization. Thus the formal clause of
Lemma~\ref{lemma:tree-gluing} applies. It gives a single nonconstant
holomorphic first integral $F$ in a neighborhood of the entire exceptional
divisor $E$ and, throughout the tree induction, preserves a factorization
\[
 \pi^*\widehat f=\widehat\Phi\circ F.
\]
This is the arbitrary-genus version of the transversally formal
construction in \cite[Chapter~V, \S1, Lemma~$3'$]{MatteiMoussu1980}.

The exceptional fiber $E=\pi^{-1}(0)$ is connected, by the connectedness
theorem for a resolution of a normal surface germ. The function $F$ is
constant on the regular part of each invariant component of $E$, and
continuity at their intersection points shows that it has a single
constant value on $E$. After subtracting this constant, we may assume
that $F|_E=0$. Since $\pi$ is biholomorphic away from $E$, the function
\[
f:=F\circ\pi^{-1}
\]
is a holomorphic first integral on a punctured neighborhood of $0$
in $X$. Finally, since $X$ is normal and $\{0\}$ has codimension two,
the Hartogs--Riemann extension theorem gives a unique extension
\[
f\in\OX_{X,0}.
\]
\end{proof}

\begin{proof}[Proof of Theorem D] We assume that $(X,0)\subset \C^m$.
The proof is a consequence of Theorem C and the
Prolongation Theorem and goes as
follows:

If $n=2$, then a codimension-one foliation on $(X,0)$ is a foliation by
curves, and the conclusion follows directly from Theorem~C.

Assume now that $n\ge3$. We consider a generic linear subspace $H\subset\C^m$ of codimension $n-2$, in general
position with respect to the pair $((X,0),\F)$, and let $j\colon H\hookrightarrow\C^m$ denote the inclusion.
\[
       (X^*,0)=j^{-1}(X,0)=(j^{-1}(X),0)=(X\cap H,0)\hookrightarrow(\C^m,0).
\]
Let $i\colon(X^*,0)\hookrightarrow(X,0)$ denote the induced inclusion and
put $\F^*=i^*(\F)$.
Choose a holomorphic $1$-form $\omega\in\Omega^1_{X,0}$ defining $\F$ on
$(X,0)$ and let $\omega^*=i^*(\omega)\in\Omega^1_{X^*,0}$. Given a formal first integral $\hat f\in\widehat\OX_{X,0}$ for
$\F$ on $(X,0)$, we have $d\hat f\wedge\omega=0$ as a formal $2$-form. Consider the pull-back $i^*\colon\Omega^1_{X,0}\to\Omega^1_{X^*,0}$
and its extension to formal differential forms. Applying the formal
pull-back $\widehat i^{\,*}$, we obtain
\[
       d\bigl(\widehat i^{\,*}\widehat f\bigr)\wedge i^*\omega
       =\widehat i^{\,*}(d\widehat f\wedge\omega)=0.
\]
Consequently,
\[
d\bigl(\widehat i^{\,*}\widehat f\bigr)\wedge\omega^*=0.
\]
After subtracting the constant term, we may assume that
\[
\widehat f(0)=0.
\]
Let
\[
r:=\ord_0(\widehat f)
\]
and denote by
\[
\operatorname{in}_r(\widehat f)
\in\mathfrak m_0^r/\mathfrak m_0^{r+1}
\]
its first nonzero homogeneous term. In addition to the preceding
genericity conditions, we choose $H$ so that
\[
\operatorname{in}_r(\widehat f)
\big|_{C_0(X)\cap H}\not\equiv0.
\]
Here the tangent cone is understood scheme-theoretically. Put
\[
g:=\operatorname{in}_r(\widehat f)
\in\operatorname{gr}_{\mathfrak m_0}\OO_{X,0}.
\]
By hypothesis, the tangent cone $C_0(X)$, or equivalently
$\operatorname{gr}_{\mathfrak m_0}\OO_{X,0}$, is reduced. Hence the nonzero
homogeneous element $g$ does not vanish identically on at least one
irreducible component of $C_0(X)$. A generic two-dimensional linear section
meeting that component is therefore not contained in the zero set of $g$.
Consider the incidence subset
\[
\mathcal B_g:=
\left\{
H\in\operatorname{Gr}(m-n+2,m):
g|_{C_0(X)\cap H}=0
\right\}.
\]
This is consequently a Zariski-closed proper subset of the Grassmannian.
Indeed, on the universal incidence family over the Grassmannian,
restriction of $g$ is a section of the appropriate homogeneous sheaf; the
locus on which this section vanishes identically on the fibre is closed by
upper semicontinuity of the rank of the corresponding finite-dimensional
restriction map.  It is proper because a component on which $g$ is
nonzero admits a linear section meeting its nonvanishing locus.
Thus the complement of
$\mathcal B_g$ is a nonempty Zariski-open set. Intersecting it with the
open sets expressing the normal-tree, transversality, and
quotient-prolongation-admissibility conditions, we may impose all four
requirements simultaneously, because the Grassmannian is irreducible and
a finite intersection of nonempty Zariski-open subsets is nonempty.

The pull-back
\[
\widehat f^{\,*}:=\widehat i^{\,*}\widehat f\in\widehat\OX_{X^*,0}
\]
is a formal first integral for the foliation defined by $\omega^*$ on
$(X^*,0)$, i.e., for the pull-back foliation
$\F^*=i^*(\F)$. The transversality condition guarantees that
$\F^*$ is a foliation by curves. Moreover, the condition on the
initial term implies that
\[
\widehat f^{\,*}\neq0.
\]
Since $\widehat f^{\,*}$ has zero constant term, it is nonconstant.
Applying Theorem~\ref{thm:C}, we conclude that $\F^*$ has a holomorphic first
integral on $(X^*,0)$. Invoking the Prolongation
Theorem~\ref{thm:E}, we obtain a
holomorphic first integral for $\F$ on $(X,0)$.
\end{proof}

\end{document}